\documentclass[11pt,a4paper]{article}
\usepackage{amsmath,amssymb,amsthm,fullpage,graphicx,pstricks,xcolor}

\numberwithin{equation}{section}
\newcommand{\R}{\mathbb{R}}
\newcommand{\E}{\mathbf{E}}
\newcommand{\Z}{\mathbb{Z}}
\newcommand{\I}{\mathbb{I}}
\newcommand{\pp}{\mathbf{P}}

\newcommand{\kP}{\mathcal{P}}

\newcommand{\kE}{\mathcal{E}}
\newcommand{\supp}{\textrm{supp}}
\newcommand{\cov}{\mathrm{Cov}}
\newcommand{\Var}{\mathrm{Var}}

\newtheorem{lem}{Lemma}[section]

\newtheorem{prop}[lem]{Proposition}
\newtheorem{theo}[lem]{Theorem}

\newtheorem{cor}[lem]{Corollary}

\newtheorem{rem}[lem]{Remark}

\begin{document}

\title{\vspace{20pt}Spectral dimension of simple random walk\\on a long-range percolation cluster}
\author{V.~H.~Can\footnote{\scriptsize{  Institute of Mathematics, Vietnam Academy of Science and Technology, 18 Hoang Quoc Viet, Cau Giay, Hanoi, Vietnam, \& Department of Statistics and Data Science, National University of Singapore, 6 Science Drive 2 Singapore 117546, cvhao89@gmail.com.}}, D.~A.~Croydon\footnote{\scriptsize{Research Institute for Mathematical Sciences, Kyoto University, Kyoto 606-8502, Japan, croydon@kurims.kyoto-u.ac.jp.}} and T.~Kumagai\footnote{\scriptsize{Department of Mathematics, Faculty of Science and Engineering, Waseda University, 3-4-1 Okubo, Shinjuku-ku, Tokyo 169-8555, Japan, t-kumagai@waseda.jp.}}}

\maketitle

\begin{abstract}
Consider the long-range percolation model on the integer lattice $\mathbb{Z}^d$ in which all nearest-neighbour edges are present and otherwise $x$ and $y$ are connected with probability $q_{x,y}:=1-\exp(-|x-y|^{-s})$, independently of the state of other edges. Throughout the regime where the model yields a locally-finite graph, (i.e.\ for $s>d$,) we determine the spectral dimension of the associated simple random walk, apart from at the exceptional value $d=1$, $s=2$, where the spectral dimension is discontinuous. Towards this end, we present various on-diagonal heat kernel bounds, a number of which are new. In particular, the lower bounds are derived through the application of a general technique that utilises the translation invariance of the model. We highlight that, applying this general technique, we are able to partially extend our main result beyond the nearest-neighbour setting, and establish lower heat kernel bounds over the range of parameters $s\in (d,2d)$. We further note that our approach is applicable to short-range models as well.
\end{abstract}

%60K37 (2000-now) Processes in random environments
%35K05 (1973-now) Heat equation
%60J15 (1973-1999) Random walks
%60J35 (1973-now) Transition functions, generators and resolvents
%60J74 (2020-now) Jump processes on discrete state spaces
%82B43 (1991-now) Percolation

\section{Introduction}

The study of random walks on percolation clusters on the integer lattice $\mathbb{Z}^d$ goes back a long way, at least as far as de Gennes' 1976 description of such a process as an `ant in a labyrinth' \cite{dG}. Mathematically, diffusive scaling limits were first established with respect to the so-called annealed/averaged law, under which both the random process and environment are integrated out \cite{DFGW}. More recently, building on the Gaussian heat kernel estimates of \cite{MTB}, scaling limits under the quenched law (that is, for typical realisations of the environment) have also been obtained \cite{BB,MP,SS}. When the random walk is strongly recurrent, some general theory has been established to obtain on-diagonal heat kernel estimates, and such methods have been used to identify the spectral dimension,
which is the exponent governing the on-diagonal decay of the heat kernel, of the
random walk on critical percolation clusters conditioned to be infinite (see for example \cite{BJKS,Bar-K,KozN,KM,KumSF}).
Whilst the works cited so far have dealt with the nearest-neighbour case, in which only edges between points in $\mathbb{Z}^d$ a unit Euclidean distance apart are considered, it is natural to generalise the model to allow the possibility of edges spanning arbitrarily large distances. In the last decade, substantial progress has been made in understanding random walks on such long-range percolation models, most notably in \cite{CS,CS2}, some of the main results of which are recalled below. Our contribution in this paper is twofold:
\begin{enumerate}
  \item[(i)] To give general sufficient conditions for an on-diagonal lower bound of the heat kernel for random walk
under both the quenched and annealed laws on stationary random media (see Theorem \ref{abth} and Corollary \ref{abthcor} below);
  \item[(ii)] To determine the spectral dimension of random walk on a long-range percolation cluster, under both the quenched and annealed laws, throughout (almost) the entire range of parameters for which the model is defined
(see Theorems \ref{lrpq}, \ref{lrpa} and Corollary \ref{corlrp}).
\end{enumerate}

Concerning (i), we note that the previous work in \cite{BJKS,KM,KumSF} applies only for strongly recurrent random walks,
and the sufficient conditions the latter articles describe for heat kernel lower bounds are rather complicated. Instead, we use stationarity of the model and a useful estimate from  \cite[Theorem 3.7]{L}, which was motivated by the problem of understanding the behaviour of the random walk on certain planar random graphs, such as the uniform infinite planar triangulation/quadrangulation. Details are discussed in Section \ref{hkesec}. Although the techniques we develop are principally targeted at understanding random walk on long-range percolation clusters, we note they are also applicable to short-range models. As a basic example of such, we discuss their use for studying the random walk on the integer lattice in Section \ref{zdsec}.

To present the background and results concerning (ii) more precisely, we proceed to introduce the main application of interest in this paper. Specifically, we consider a long-range percolation model with vertex set $\Z^d$, where $d\geq 1$. For simplicity, in the introduction we suppose that all nearest-neighbour edges are present, although we will later discuss a generalisation of this. For any $x,y\in\mathbb{Z}^d$ with $|x-y|>1$, we suppose the edge between them appears with probability
\begin{equation}\label{pxy}
q_{x,y}=1-\exp(-|x-y|^{-s}),
\end{equation}
independently of the state of other edges. (For most of the subsequent discussion, we could weaken the tail assumption on $q_{x,y}$. Indeed, for our heat kernel estimates in Theorems \ref{lrpq} and \ref{lrpa} below, it would be enough to assume that $c_1|x-y|^{-s}\leq q_{x,y} \leq c_2 |x-y|^{-s}$. We choose to restrict to the specific choice of $q_{x,y}$ above simply for convenience.) The parameter $s$ is called the exponent of the long-range percolation model. In order to ensure that each vertex is directly connected by an edge to only a finite number of vertices (as is required to define the associated discrete-time random walk), we assume that $s$ takes a value strictly greater than $d$. We denote the resulting random graph by $G=(V,E)$, and take the root $\rho$ to be the origin in $\Z^d$. Moreover, we will use the notation LRP($d,s$) to represent this model, and suppose it is built on a probability space with probability measure $\mathbf{P}$ and expectation $\mathbf{E}$.

Providing some context for our results, the following summarises scaling limits that are known to hold for the discrete-time simple random walk on LRP($d$,$s$). Given the environment $G$, this process, which we will denote by $(X_n)_{n\geq 0}$, jumps on each time step from its current location to a uniformly-chosen neighbour in the graph $G$. In the subsequent theorem, it is further assumed that $X_0=\rho$. Part (a), which concerns the stable regime, was established in \cite[Theorem 1.1]{CS2}. As for the Gaussian regime of part (b), the $d=1$ case was dealt with in \cite[Theorem 1.2]{CS2} (see also \cite{ZZ}), and the $d\geq 2$ case in \cite{BCKW}. (We give a new argument for $d=1$, $s>2$ in Section \ref{qipd1} below.) Note that both the stable and Gaussian regimes are thought to be incomplete (see discussion in \cite{CS,CS2} and \cite[Problem 2.9]{BCKW}), and our heat kernel bounds support conjectures about how they extend. Figure \ref{fig1} gives a graphical overview of the situation.

\begin{theo}[Long-range percolation, scaling limits, \cite{BCKW,CS2,ZZ}]\label{lrps}\hspace{10pt}\\
(a) If $d\geq 1$ and $s\in(d,d+1)$, then for $\mathbf{P}$-a.e.\ realisation of LRP($d,s$) and every $p\in [1,\infty)$, the law of
\[\left(n^{-\frac{1}{s-d}}X_{nt}\right)_{t\in[0,1]}\]
on $L^p([0,1])$ converges weakly to the law of an isotropic $\alpha$-stable L\'{e}vy process with $\alpha=s-d$.\\
(b) If $d\geq 1$ and $s>2d$, then for $\mathbf{P}$-a.e.\ realisation of LRP($d,s$), the law of
\[\left(n^{-\frac{1}{2}}X_{nt}\right)_{t\geq 0}\]
on $C([0,\infty))$ converges weakly to that of $(B_{\sigma^2 t})_{t\geq 0}$, where $(B_{t})_{t\geq 0}$ is standard Brownian motion on $\mathbb{R}^d$, and $\sigma^2\in(0,\infty)$ is a deterministic constant.
\end{theo}

\begin{rem}\label{rem1} (i) To make the statement of the above theorem completely accurate, we need to describe a convention for determining the value of $(X_t)_{t\geq 0}$ between integer times. For both parts above, the result would hold if one were to do this by linear interpolation. Alternatively, using the $J_1$ topology on the Skorohod space $D([0,\infty))$ for part (b) above, one could consider $X_{\lfloor t\rfloor}$ in place of $X_t$. In this article, we will henceforth adopt the latter approach; that is, if we write a continuous variable, $x$ say, where a discrete
argument is required, we suppose it should be treated as $\lfloor x\rfloor$.\\
(ii) Note that part (a) above is proved in \cite{CS2} without the assumption of nearest-neighbour edges being present, which requires a substantial amount of extra work to deal with the percolation issues involved.
\end{rem}

We next set out our heat kernel estimates for LRP($d$,$s$). Given $G$, the (quenched) heat kernel/transition density of $X$ is defined by setting
\[p^G_t(x,y):=\frac{P^G_x\left(X_t=y\right)}{\mathrm{deg}_G(y)},\qquad \forall x,y\in \mathbb{Z}^d,\:t\geq 0,\]
where $P^G_x$ is the (quenched) law of $X$ started from $X_0=x$, and $\deg_G(y)$ is the usual graph degree of $y$ in $G$. For typical realisations of the environment, we have the following bounds. The constants $c_i$ and $\delta_i$ are deterministic, and the $\delta_i$ in particular are discussed in the subsequent remark. We further highlight that, in the parameter regimes where scaling limits are known, the lower heat kernel bounds follow from a general argument adapted from \cite{Biskup} (see Lemma \ref{biskuplem} below). The main contribution of this article is in establishing the remaining lower bounds, which we do by developing \cite[Theorem 3.7]{L} (cited below as Proposition \ref{kpr}). As for the upper bounds, the result in the stable case was previously known from \cite[Theorem 1]{CS}. This was based on a general argument for checking quenched heat kernel upper bounds on random media, which we believe would also be appropriate in the Gaussian case. However, we use another argument based on comparison with a simple random walk and a time change, which more easily adapts to the annealed case, and allows us to remove the logarithmic terms there. For discussion of the case $d=1$, $s=2$, see Remark \ref{critcase} below.

\begin{theo}[Long-range percolation, quenched bounds]\label{lrpq}\hspace{10pt}\\
(a) If $d\geq1$ and $s\in(d,\min\{d+2,2d\})$, then LRP($d,s$) satisfies, $\mathbf{P}$-a.s., for all $t\in\mathbb{N}$ large enough,
\begin{equation}\label{qs}
c_1t^{-\frac{d}{s-d}}\left(\log t\right)^{-\delta_1}\leq p^G_{2t}(\rho,\rho)\leq c_2t^{-\frac{d}{s-d}}\left(\log t\right)^{\delta_2}.
\end{equation}
(b) If $d= 1$ and $s> 2$, then LRP($d,s$) satisfies, $\mathbf{P}$-a.s., for all $t\in\mathbb{N}$ large enough,
\begin{equation}\label{qg2}
c_3t^{-\frac{1}{2}}\leq p^G_{2t}(\rho,\rho)\leq
c_4t^{-\frac{1}{2}}.
\end{equation}
The upper bound holds for $d= 1$ and $s= 2$ as well. \\
(c) If $d\geq 2$ and $s\ge d+2$,
then LRP($d,s$) satisfies, $\mathbf{P}$-a.s., for all $t\in\mathbb{N}$ large enough,
\begin{equation}\label{qg}
c_5t^{-\frac{d}{2}}\left(\log t\right)^{-\delta_3}\leq p^G_{2t}(\rho,\rho)\leq
c_{6}t^{-\frac{d}{2}}\left(\log t\right)^{\delta_4}.
\end{equation}
\end{theo}

\begin{rem}\label{deltaq}
In the above result, we can take
\[\delta_1:=\left(\frac{4s}{s-d}+\varepsilon\right)\mathbb{I}_{\{s\in[d+1,d+2)\}};\]
note that this means we do not need a log term in the regime where there is a scaling limit. As indicated above, the upper bound in \eqref{qs} is essentially due to \cite{CS}, with $\delta_2$ being as given by the $\delta$ of \cite[Theorem 1]{CS}, which is not explicit. Similarly to Remark \ref{rem1} above, the result of \cite{CS} does not require nearest-neighbour bonds to be present in the model. For $s\in (d,2d)$, which includes the entire stable regime, we explain how to extend the lower bounds of Theorem \ref{lrpq} (and Theorem \ref{lrpa} below) to the non-nearest-neighbour setting in Section \ref{ss:ll}.

The article \cite{KM} gives upper and lower bounds with logs in the case $d=1$ and $s>2$; the above theorem improves on this. (In this case, the assumption of nearest-neighbour bonds is not purely for convenience -- see the discussion preceding \cite[Theorem 1.2]{CS2}.)

We further have
\[\delta_3:=\left(\frac{5d}{2}+4+\varepsilon\right)\mathbb{I}_{\{s=d+2\}}
+\left(2d+4+\varepsilon\right)\mathbb{I}_{\{d+2<s\leq2d\}};\]
again this is zero where we have a scaling limit. And also,
\[\delta_4:=\frac{d-1}{2}.\]
We do not expect any of the $\delta_i$s to be the best possible constants. For discussion concerning the removal of the log term in the upper bound of \eqref{qg} in particular, see Remark \ref{nologrem} below.
\end{rem}

As for the annealed heat kernel, which is obtained by integrating out the randomness of the environment, we have the following.

\begin{theo}[Long-range percolation, annealed bounds]\label{lrpa}\hspace{10pt}\\
(a) If $d\geq1$ and $s\in(d,\min\{d+2,2d\})$, then LRP($d,s$) satisfies, for all $t\in \mathbb{N}$,
\begin{equation}\label{as}
c_1t^{-\frac{d}{s-d}}\leq \mathbf{E}\left(p^G_{2t}(\rho,\rho)\right)\leq c_2t^{-\frac{d}{s-d}}\left(\log t\right)^{\delta_2}.
\end{equation}
(b) If $d\geq2$ and $s=d+2$, then LRP($d,s$) satisfies,  for all $t\in \mathbb{N}$,
\begin{equation}\label{ac}
c_3t^{-\frac{d}{2}}\left(\log t\right)^{-\delta_5}\leq\mathbf{E}\left(p^G_{2t}(\rho,\rho)\right)\leq c_4t^{-\frac{d}{2}}.
\end{equation}
(c) If $d\geq 1$ and $s>\min\{d+2,2d\}$, then LRP($d,s$) satisfies,  for all $t\in \mathbb{N}$,
\begin{equation}\label{ag}
c_5t^{-\frac{d}{2}}\leq\mathbf{E}\left(p^G_{2t}(\rho,\rho)\right)\leq c_6t^{-\frac{d}{2}}.
\end{equation}
The upper bound holds for $d= 1$ and $s= 2$ as well.
\end{theo}

\begin{rem} We can take the same $\delta_2$ as in Theorem \ref{lrpq}, and
\[\delta_5:=\frac{d}{2}+\varepsilon.\]
For $d=1$ and $s>2$, the bounds of \eqref{ag} are obtained in \cite[Section 2]{KM}. For \eqref{ac}, we conjecture that some log correction is necessary, i.e.\ the upper bound is not sharp. We also anticipate that the upper bound in \eqref{as} is not sharp, in that no log term is required in this case.
\end{rem}

As a straightforward consequence of Theorems \ref{lrpq} and \ref{lrpa}, we can read off the spectral dimension of  LRP($d,s$) for $d\geq 1$, $s>d$, apart from at the value $d=1$, $s=2$. Precisely, the quenched spectral dimension is defined to be the $\mathbf{P}$-a.s.\ limit
\[d^{(q)}_{s}(d,s):=-\lim_{t\rightarrow\infty}\frac{2\log p^G_{2t}(\rho,\rho)}{\log t},\]
and the corresponding annealed spectral dimension is the limit
\[d^{(a)}_{s}(d,s):=-\lim_{t\rightarrow\infty}\frac{2\log\mathbf{E}\left( p^G_{2t}(\rho,\rho)\right)}{\log t}.\]
See Figure \ref{fig2} for an illustration of the following result.

\begin{cor}[Long-range percolation, quenched and annealed spectral dimension]\label{corlrp}\hspace{10pt}\\
(a) If $d\geq1$ and $s\in(d,\min\{d+2,2d\})$, then, $\mathbf{P}$-a.s.,
\[d^{(a)}_{s}(d,s)=d^{(q)}_{s}(d,s)=\frac{2d}{s-d}.\]
(b) If $d\geq1$ and $s>\min\{d+2,2d\}$ or $d\geq 2$ and $s=d+2$, then, $\mathbf{P}$-a.s.,
\[d^{(a)}_{s}(d,s)=d^{(q)}_{s}(d,s)=d.\]
\end{cor}

\begin{rem}\label{critcase}
As shown by Corollary \ref{corlrp} (and Figure \ref{fig2}), there is a discontinuity in the spectral dimension at $d=1$, $s=2$. Whilst it might be possible to argue from the techniques of this article that, if it exists, the spectral dimension lies in the interval [1,2], determining the exact value seems highly non-trivial. See the discussion of \cite{BBd}, \cite[Problem 2.10]{BCKW}, \cite{DS} and \cite[Remark 2.3(2)]{KM} for further background on the difficulties found in this case.
\end{rem}

\begin{figure}
\begin{center}\includegraphics[width=0.4\textwidth]{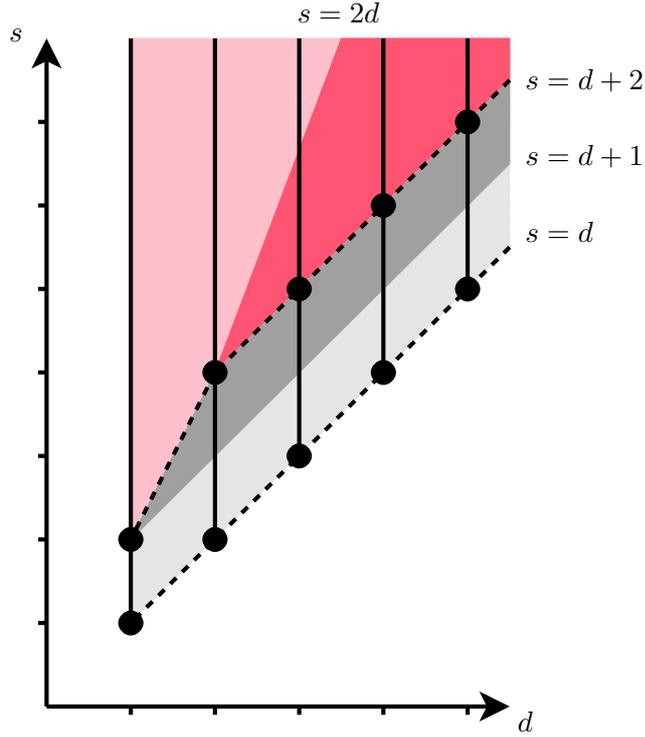}
\rput(-6.7,9.1){$s$}
\rput(-0,0){$d$}
\rput[l](0.0,8.5){$s=d+2$}
\rput[l](0.0,7.5){$s=d+1$}
\rput[l](0.0,6.5){$s=d$}
\rput[l](-3,9.4){$s=2d$}
\end{center}
\caption{Heat kernel regimes for the long-range percolation model studied in this article. The two pink regions correspond to the Gaussian regime, that is $s>\min\{d+2,2d\}$. Gaussian scaling limits are known to hold in the lighter pink region ($s>2d$). Our results provide new heat kernel lower bounds in the darker pink region ($d\geq 2$ and $s\in (\min\{d+2,2d\},2d]$), and establish that the spectral dimension is $d$ throughout both regions. The two grey regions correspond to the stable regime, that is $s\in (d,\min\{d+2,2d\})$. Stable scaling limits are known to hold in the lighter grey region ($s\in(d,d+1)$). We give new heat kernel lower bounds in the darker grey region ($s\in[d+1,d+2)$), and establish that the spectral dimension of the model is given by $2d/(s-d)$ throughout both regions. As confirmed by Corollary \ref{corlrp}, when $d=1$, there is a discontinuity in the spectral dimension at $s=2$. Apart from when $d=1$, we also provide estimates along the critical line $s=\min\{d+2,2d\}$.}\label{fig1}
\end{figure}

\begin{figure}
\begin{center}\includegraphics[width=0.8\textwidth]{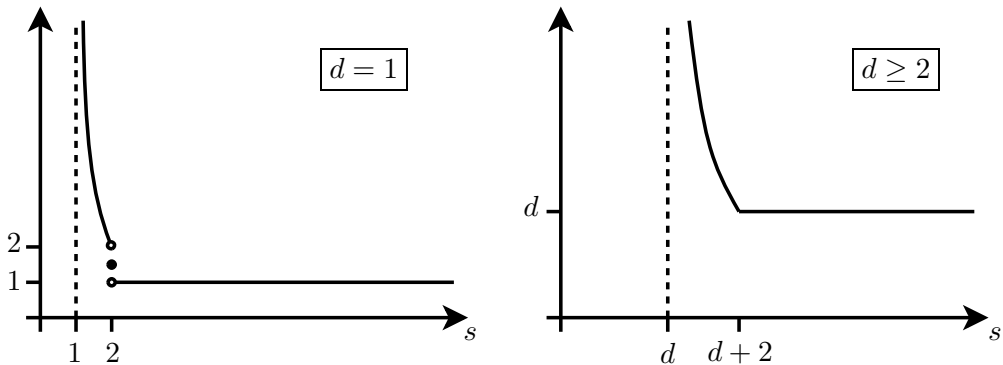}
\rput(-0.1,0){$s$}
\rput(-8.4,3.5){\fbox{$d=1$\vphantom{$d\geq2$}}}
\rput(-1.4,3.5){\fbox{$d\geq 2$}}
\rput(-12.2,-.25){$1$}
\rput(-11.7,-.25){$2$}
\rput(-13.0,.7){$1$}
\rput(-13.0,1.2){$2$}
\rput(-7,0){$s$}
\rput(-4.4,-.25){$d$}
\rput(-3.45,-.25){$d+2$}
\rput(-6.2,1.7){$d$}
\end{center}
\caption{Spectral dimension as a function of $s$ for $d=1$ (left) and $d\geq 2$ (right). The existence and value of the spectral dimension is not yet known for $d=1$, $s=2$, see Remark \ref{critcase} for discussion of this case.}\label{fig2}
\end{figure}

The remainder of the article is organised as follows. In Section \ref{hkesec}, we present our general approach for establishing quenched and annealed heat kernel lower bounds on random media, see Theorem \ref{abth} and Corollary \ref{abthcor}, and also discuss the related upper bound of \cite{CS}. The various assumptions required to apply these results are checked for long-range percolation in Section \ref{sec3}. Then, in Section \ref{sec4}, we put the pieces together to deduce the lower heat kernel bounds of Theorems \ref{lrpq} and \ref{lrpa}, and also give an argument for the corresponding upper heat kernel bounds. Finally, Section \ref{oqsec} lists some questions left open by this work, and Section \ref{secA} is an appendix in which we: explain how the lower heat kernel bound applies to the simpler setting of random walk on $\mathbb{Z}^d$; describe how a quenched scaling limit automatically implies a quenched lower heat kernel bound; present an alternative proof of a quenched invariance principle for long-range percolation in the one-dimensional setting; and describe an extension of our lower heat kernel bounds to a long-range percolation model in which non-nearest-neighbour bonds are not necessarily present.

Concerning notational conventions, we write $x\wedge y:=\min\{x,y\}$ and $x\vee y:=\max\{x,y\}$. For non-negative sequences $f(n)$ and $g(n)$, we define $f(n)\asymp g(n)$ to mean that there exist strictly positive constants $c_1$ and $c_2$ such that $c_1f(n)\leq g(n)\leq c_2f(n)$, and $f(n)\preceq g(n)$ to mean that there exists a strictly positive constant $c$ such that $f(n)\leq cg(n)$. We write $c,C$ for deterministic constants that might change value from line to line.

\section{Heat kernel estimates}\label{hkesec}

\subsection{Lower heat kernel bound}

Towards establishing our lower heat kernel bounds, we present a result from \cite{L} that shows a lower heat kernel bound must hold on some proportion of vertices in a graph, as determined by the sizes and capacities of the pieces in a suitable decomposition of the graph. The latter result is given for an arbitrary connected, finite graph $G =(V,E)$, where $V$ is a set of vertices and $E$ is a set of bonds. We set $\tilde \pi (x)= \deg_G(x)$ for all $x \in V$, and define $\pi$ to be a version of $\tilde \pi$ normalised to be a probability measure, namely
\[\pi(x) = \frac{\deg_G(x)}{2|E|},\qquad \forall x \in V,\]
where for a set $A$, we denote by $|A|$ the number of elements of $A$. Note that $\pi$ is the stationary probability measure for the discrete-time simple random walk associated with $G$. We further write, for any $\varepsilon >0$,
\begin{equation} \label{pst}
\pi^{\star}(\varepsilon) = \max \left\{\pi(W): |W| \leq \varepsilon |V| \right\}.
\end{equation}
The (on-diagonal part of the) natural Dirichlet form on $G$ is defined for functions $\psi:V\rightarrow\mathbb{R}$ as follows:
\begin{equation}\label{edef}
\kE(\psi)=\frac 12\sum_{x,y:\:x\sim y} (\psi(x)-\psi(y))^2,
\end{equation}
where we write $x\sim y$ to mean that $x$ and $y$ are connected by an edge in $E$.
 As a final piece of notation needed to state the result of \cite{L}, let us call a pair $(A,\Omega)$ of subsets $A\subseteq \Omega\subseteq V$ a capacitor, and define the capacity of $(A,\Omega)$ by
\[\mathrm{cap}_\Omega(A):=\inf_{\psi:V:\rightarrow[0,1]}\left\{2 \kE(\psi):\:\psi|_A=1,\:\supp(\psi)\subseteq\Omega\right\},\]
where $\supp(\psi):=\{x\in V:\:\psi(x)\neq 0\}$. Note that our definition of capacity differs from the definition in \cite{L} by a factor of $2|E|$.

\begin{prop}[{\cite[Theorem 3.7]{L}}]\label{kpr}
Let $G=(V,E)$ be a connected, finite graph, and $M\in(0,\infty)$ a constant. Suppose that for some $k\leq |V|$, there are capacitors $(A_1,\Omega_1),\dots, (A_k,\Omega_k)$ such that $(\Omega_i)_{i=1}^k$ are pairwise disjoint and $|\Omega_i|\leq M$ for all $i=1,\dots,k$. Then, for all $\varepsilon>0$ and $t\in \mathbb{N}$,
\[\pi \left(\left\{x \in V:\: p^G_{2t} (x,x) \geq \frac{\varepsilon |V|}{8M|E|}\right\}\right) \geq -2 \pi^{\star}(\varepsilon) + \sum_{i=1}^{k}\pi(A_i)-\frac t{|E|}\sum_{i=1}^k\mathrm{cap}_{\Omega_i}(A_i).\]
\end{prop}

\begin{rem}
The appearance of the above result in \cite{L} is slightly different, as it is expressed in terms of the transition probability, rather than the heat kernel/transition density.
\end{rem}

We will apply this result to sequences of graphs that converge in the sense of Benjamini-Schramm \cite{BS}. In particular, for the limit of the sequence, we take $G$ to be a random connected, locally-finite graph, rooted at a distinguished vertex $\rho$. For each $n\geq 1$, $G_n$ will be a random connected, finite graph, rooted at a uniformly chosen vertex $\rho_n$. Given any $k\geq 1$, we suppose that the ball in $G_n$ of radius $k$ (according to the usual shortest path graph distance) centred at the root $\rho_n$ converges in distribution to the ball in $G$ of radius $k$ centred at the root $\rho$ (see \cite[Section 1.2]{BS} for details). To apply Proposition \ref{kpr} to such a sequence, for given non-negative constants $\alpha$, $(\delta_i)_{i=0}^3$ and $\gamma$, and deterministic function $\lambda: (0,\infty)\to [1,\infty)$, we consider the following conditions (that are in fact relevant to any sequence of connected, finite graphs).
\begin{description}
\item [(A1)] For all $n\in \mathbb{N}$,
$$\E \left(\frac{|E_n|^2}{|V_n|^2}\right) \leq \alpha,$$
where $E_n$ and $V_n$ are the edge and vertex set of $G_n$, respectively.
\item [(A2)] For all $\varepsilon >0$, if $n$ is large enough, then
$$ \E\left(\pi_n^{\star}(\varepsilon)^2\right) \leq \alpha \varepsilon,$$
where $\pi_n^{\star}$ is defined as at \eqref{pst} from $\pi_n$, the stationary probability measure of simple random walk on $G_n$.
\item[(A3)] For each $t\in\mathbb{N}$, there exists an integer $n_0=n_0(t)$ such that for each $n\geq n_0$, with probability at least $1-\lambda(t)^{-\delta_0}$, there are $k\leq |V_n|$ capacitors $(A_1,\Omega_1),\dots, (A_k,\Omega_k)$ for the graph $G_n$ such that $(\Omega_i)_{i=1}^k$ are pairwise disjoint and also
\begin{itemize}
\item[(a)] $\max_{i=1,\dots,k}|\Omega_i|\leq \alpha t^\gamma\lambda(t)^{\delta_1}$;
\item[(b)] $\sum_{i=1}^k \pi_n(A_i) \geq 1 -\alpha \lambda(t)^{-\delta_2}$;
\item[(c)] $\sum_{i=1}^k\mathrm{cap}_{\Omega_i}(A_i)\leq 2\alpha |E_n|t^{-1} \lambda(t)^{-\delta_3}$.
\end{itemize}		
\end{description}
We are now ready to state the main result of this section. The assumptions (A1), (A2) and (A3) are clearly designed to feed into the bound of Proposition \ref{kpr}.

\begin{theo} \label{abth}		
Assume that $(G,\rho)$ is the Benjamini-Schramm limit of the sequence $(G_n,\rho_n)$, $n\geq 1$, which satisfies (A1), (A2), (A3) for some positive constants $\alpha$, $(\delta_i)_{i=0}^3$ and $\gamma$, and deterministic function $\lambda: (0,\infty)\to [1,\infty)$. There then exists a constant $C_\alpha$ only depending on $\alpha$ such that, for all $t\in\mathbb{N}$ large enough,
\[\pp \left(p^G_{2t}(\rho,\rho) \geq \frac{1}{t^{\gamma}\lambda(t)^{{\delta_1+2(\delta_2\wedge\delta_3)}} }\right) \geq 1 - C_\alpha\lambda(t)^{-\frac{\delta_0\wedge\delta_2\wedge \delta_3}{2}}.\]
\end{theo}
\begin{proof} Since $(G_n, \rho_n)$ converges in a Benjamini-Schramm sense to $(G,\rho)$, and $p_{2t}^{G}(\rho,\rho)$ only depends on the ball of radius $2t$ about $\rho$, it holds that, for each fixed $t$,
\[\lim\limits_{n \rightarrow \infty}	\pp \left(p^{G_n}_{2t}(\rho_n, \rho_n) < \frac{1}{t^{\gamma}\lambda(t)^{{\delta_1+2(\delta_2\wedge\delta_3)}}  }\right) = \pp \left(p^G_{2t}(\rho,\rho) < \frac{ 1}{t^{\gamma} \lambda(t)^{{\delta_1+2(\delta_2\wedge\delta_3)}} }\right).\]
Hence, for each fixed $t$ and $\delta$, if $n$ is large, then
\begin{equation} \label{ke0}
\left|	\pp \left(p_{2t}^{G_n}(\rho_n, \rho_n) < \frac{  1 }{t^{\gamma}\lambda(t)^{{\delta_1+2(\delta_2\wedge\delta_3)}} }\right) - \pp \left(p_{2t}^G(\rho,\rho) < \frac{1}{t^{\gamma} \lambda(t)^{{\delta_1+2(\delta_2\wedge\delta_3)}} }\right) \right | \leq \lambda(t)^{-{\delta}}.
\end{equation}
Moreover, writing $\mathbb{I}$ for an indicator function,
\begin{eqnarray}
\lefteqn{\pp \left(p_{2t}^{G_n}(\rho_n,\rho_n) < \frac{1}{t^{\gamma} \lambda(t)^{{\delta_1+2(\delta_2\wedge\delta_3)}} }\right)}\nonumber\\
 & = &\E \left[ \pp \left(p_{2t}^{G_n}(\rho_n,\rho_n) < \frac{1}{t^{\gamma}\lambda(t)^{{\delta_1+2(\delta_2\wedge\delta_3)}}  } \,\vline \, G_n\right)\right] \notag \\
&=& \E \left[ \frac{1}{|V_n|} \sum_{x \in V_n} \I \left(\left\{ p_{2t}^{G_n}(x,x) < \frac{1}{t^{\gamma} \lambda(t)^{{\delta_1+2(\delta_2\wedge\delta_3)}} }\right\}\right)\right] \notag \\
& \leq & \E \left[ \frac{2 |E_n|}{|V_n|}  \sum_{x \in V_n} \frac{\deg_{G_n}(x)}{2|E_n|}\I \left( p_{2t}^{G_n}(x,x) < \frac{1}{t^{\gamma} \lambda(t)^{{\delta_1+2(\delta_2\wedge\delta_3)}} }\right)\right] \notag \\
&=& 2 \E \left[ \frac{ |E_n|}{|V_n|}  \pi_n \left(\left\{ x \in V_n:  p_{2t}^{G_n}(x,x) < \frac{1}{t^{\gamma}\lambda(t)^{{\delta_1+2(\delta_2\wedge\delta_3)}} }\right\}\right)\right]  \notag\\
& \leq & 2 \E \left[\frac{|E_n|^2}{|V_n|^2}\right]^{1/2} \E \left[ \pi_n \left(\left\{ x \in V_n:  p_{2t}^{G_n}(x,x) < \frac{1}{t^{\gamma}\lambda(t)^{{\delta_1+2(\delta_2\wedge\delta_3)}}  }\right\}\right) ^2\right]^{1/2} \notag \\
& \leq & 2 \sqrt{\alpha} \E \left[ \pi_n \left( \left\{x \in V_n:  p_{2t}^{G_n}(x,x) < \frac{1}{t^{\gamma} \lambda(t)^{{\delta_1+2(\delta_2\wedge\delta_3)}} }\right\}\right) ^2\right]^{1/2}, \label{ke1}
\end{eqnarray}	
where we have applied (A1) to deduce the final inequality. Now, suppose $n_0=n_0(t)$ is an integer as in (A3). Applying (A3) and Proposition \ref{kpr} with $M=\alpha t^\gamma\lambda(t)^{\delta_1}$ and $\varepsilon = 8\alpha\lambda(t)^{-(\delta_2\wedge\delta_3)}$,  we obtain that, for $n\geq n_0$, on an event of probability at least $1-\lambda(t)^{-\delta_0}$,
\[\pi_n \left( \left\{x \in V_n:  p_{2t}^{G_n}(x,x) < \frac{1}{t^{\gamma} \lambda(t)^{{\delta_1+2(\delta_2\wedge\delta_3)}} }\right\}\right)\mathbb{I}\left(\frac{|E_n|}{|V_n|}\leq \lambda(t)^{\delta_2\wedge\delta_3} \right) \leq 2 \pi_n^{\star} (\varepsilon) + C_\alpha \lambda(t)^{-(\delta_2\wedge\delta_3)},\]
where $C_\alpha$ is a constant that only depends on $\alpha$. Therefore, by (A1) and (A2), for large $n$,
\begin{eqnarray}
\lefteqn{\E \left[\pi_n \left( \left\{x \in V_n:  p_{2t}^{G_n}(x,x) < \frac{1}{t^{\gamma} \lambda(t)^{{\delta_1+2(\delta_2\wedge\delta_3)}} }\right\}\right)  ^2 \right]^{1/2}}\nonumber\\
 &\leq& C_\alpha'\left( \E [\pi_n^{\star} (\varepsilon)^2] +\lambda(t)^{-2(\delta_2\wedge\delta_3)}
  +\lambda(t)^{-\delta_0}\right)^{1/2} \nonumber \\
&\leq & C_\alpha''\lambda(t)^{-\frac{\delta_0\wedge\delta_2\wedge \delta_3}{2}}. \label{ke2}
\end{eqnarray}
Combining \eqref{ke0}, \eqref{ke1} and \eqref{ke2}, we get that, for all $t$ large enough,
\[\pp \left(p^G_{2t}(\rho,\rho)< \frac{1}{t^{\gamma}\lambda(t)^{\delta_1+2(\delta_2\wedge\delta_3)} }\right) \leq  C_\alpha'''\lambda(t)^{-\frac{\delta_0\wedge\delta_2\wedge \delta_3}{2}},\]
which implies the desired result.
\end{proof}
\begin{rem}
When the assumption (A3) is hard to check for {\it all} integers $n$, we could  replace  (A3)  by (A3'): There exists an integer $n_1=n_1(t)$ such that for $n=n_1$, the conditions (A3)(a)-(c) hold and, moreover,
    \[\pp \left(p_{2t}^{G_n}(\rho_n, \rho_n) < \frac{1}{t^{\gamma}\lambda(t)^{{\delta_1+2(\delta_2\wedge\delta_3)}} } \right)  \leq  \pp \left(p_{2t}^G(\rho,\rho) < \frac{2}{t^{\gamma}\lambda(t)^{{\delta_1+2(\delta_2\wedge\delta_3)}} }\right)  + \lambda(t)^{-{\delta_0}}.\]
The above condition gives a replacement of \eqref{ke0}; the other parts of the proof of Lemma \ref{abth} are exactly the same.
\end{rem}

\begin{rem}
We note that Benjamini-Schramm convergence is a key input into the proof, allowing us to transfer a heat kernel estimate from a large, but unspecified, set to a single specified point. This is somewhat analogous to the approach used to understand the on-diagonal part of the annealed heat kernel of Brownian motion on stable trees in \cite{CHspec,CHspec2}, whereby random re-rooting was used to obtain point-wise asymptotics for the heat kernel from the asymptotics of the trace of the heat semigroup, which is typically a smoother object.
\end{rem}

To complete the section, we give a corollary that explains how the distributional bound of Theorem \ref{abth} can be applied to yield quenched and annealed lower heat kernel bounds.

\begin{cor}\label{abthcor}
Suppose that $(G,\rho)$ is the Benjamini-Schramm limit of the sequence $(G_n,\rho_n)$, $n\geq 1$.\\
(a) Assume that (A1),(A2), (A3) hold for some positive constants $\alpha$, $(\delta_i)_{i=0}^3$ and $\gamma$, and deterministic function $\lambda: (0,\infty)\to [1,\infty)$ satisfying $\sum_{i=1}^\infty \lambda(e^i)^{-\frac{\delta_0\wedge\delta_2\wedge \delta_3}{2}}<\infty$ and $c_1\le \lambda(t)/\lambda(e^{i+1})$ for all $e^i\le t\le e^{i+1}$ and all $i\in {\mathbb N}$. Then, $\mathbf{P}$-a.s., for all $t\in\mathbb{N}$ large enough
\[p^G_{2t}(\rho,\rho) \geq \frac{c_2}{t^{\gamma}\lambda(t)^{\delta_1+2(\delta_2\wedge\delta_3)} }.\]
(b) Assume that (A1),(A2), (A3) hold for some positive constants $\alpha$, $(\delta_i)_{i=0}^3$ and $\gamma$, and constant function $\lambda: (0,\infty)\to [1,\infty)$ given by $\lambda(t)=\lambda_0$, where $1-C_\alpha \lambda_0^{-\frac{\delta_0\wedge\delta_2\wedge \delta_3}{2}}>0$. Then
\[\mathbf{E}\left(p^G_{2t}(\rho,\rho)\right) \geq \frac{c(\alpha,\lambda_0)}{t^{\gamma}},\qquad \forall t\in\mathbb{N},\]
where $c(\alpha,\lambda_0)$ is a constant depending on the values of $\alpha$ and $\lambda_0$.
\end{cor}
\begin{proof}
(a) Given Theorem \ref{abth}, this follows from the Borel-Cantelli lemma and the monotonicity of the on-diagonal part of the heat kernel $p^G_{2t}(\rho,\rho)$. (That the latter quantity is decreasing in $t$ can be seen from
\cite[Lemma 4.1]{BarHK} and
\cite[Equation (15)]{LLT}, for example.)\\
(b) This is an obvious consequence of Theorem \ref{abth}.
\end{proof}

\subsection{Upper heat kernel bound}

In \cite[Lemma 3.1]{CS}, a general heat kernel upper bound was given. For comparison with the approach of the previous subsection, we summarize it here. In both \cite{CS} and this article, a key aspect of the required input is that we can decompose a large part of the graph in question into suitably-sized pieces that behave well in some way. Here, the focus is on the capacity of the pieces; in \cite{CS}, it is their spectral gap that plays a central role.

To present the setting of \cite[Section 3]{CS}, let $G$ be a connected, locally-finite graph, and $(Y_t)_{t\geq 0}$ be the continuous-time random walk on $G$ with unit mean holding times. For any connected, finite subgraph $H\subseteq G$ and any vertex $x \in H$, we denote by $\deg_H(x)$ the degree of $x$ within $H$. The stationary measure of the random walk on $H$ is given by
\[\pi_H(x) =\frac{\deg_H(x)}{\sum_{y\in H} \deg_H(y)}.\]
Moreover, the spectral gap of $H$ is defined as
\[{\rm Gap}_H= \inf\left\{\frac{\sum_{x,y\in H:\:x\sim y}(f(x)-f(y))^2}{2\sum_{x\in H}f(x)^2\pi_H(x)}:\:f:H\rightarrow\mathbb{R}\text{ non-constant},\:\sum_{x\in H}f(x)\pi_H(x)=0\right\}.\]
For given constants $T_1 \leq T_2$ and $\gamma>0$, and distinguished vertex $\rho\in G$, the following assumption is then considered in \cite{CS}. There exist
\begin{itemize}
  \item two positive functions $\lambda_s, V_s: s\in [T_1/2,T_2/2] \mapsto \R_+$ such that $\lambda_s$ is decreasing and $V_s$ is increasing,
  \item a family of universal constants $\{c_i,C_i\}_{1\leq i \leq 4}$ and $\delta_1$,
  \item for each $s\in [T_1/2,T_2/2]$, a distinguished connected set $B(s)$ containing $\rho$ and a partition $\kP_s$ of $B(s)$ into connected sets $\{H: H \in \kP_s\}$,
\end{itemize}
such that the following holds:
\begin{itemize}
	\item [(B1)] for all $H\in \kP_s$,
	$${\rm Gap}_H \geq \lambda_s,$$
	\item[(B2)] for all $H\in \kP_s$,
	$$c_1 V_s \leq |H| \leq C_1 V_s,$$
	\item[(B3)]  $$c_2 V_s^{-1/\gamma}\log^{-\delta_1} V_s \leq \lambda_s,$$
	\item[(B4)] $$\Delta_{\kP_s} :=\inf_{x\in B(s)}\frac{\deg_H(x)}{\deg_G(x)} \geq c_3 \log^{-1} V_s, $$
	\item[(B5)] $$P_\rho^G(Y_s\in B(s)^c)\leq C_3 V_s^{-1} \log V_s,$$
	\item[(B6)] $$1+C_3+1/(c_1c_3) \leq \psi_s(\rho) V_s \log^{-1} V_s \leq C_4, $$
where $\psi_s(\rho)=P_\rho^G(Y_{2s}=\rho)/\deg_G(\rho)$.
\end{itemize}
Applying these, the following result is proved in \cite{CS}.

\begin{theo} \cite[Lemma 3.1]{CS}
Assume the conditions (B1)--(B6) are satisfied. Then for $\delta = \delta_1 + 1/\gamma$ and $t\in [T_1,T_2]$, we have that
	\[	\psi_t(\rho) \leq \psi_{T_1}(\rho) \wedge C_5(1+C_6(t-T_1))^{-\gamma} |\log(1+C_6(t-T_1))|^{\delta\gamma},\]
where $C_5$ and $C_6$ are universal constants.	
\end{theo}

Roughly speaking, the above theorem states that if, for all $s$ large enough, we can find a connected subgraph $B(s) \subseteq G$ containing $\rho$ and a partition of this, $\{H:H\in\kP_s\}$, such that:
\begin{itemize}
	\item [(B1')] for all $H$,
	$$|H|\asymp s^{\gamma} \log^{-\delta_1} s,$$
	\item[(B2')] for all $H$,
	$${\rm Gap}_{H}\succeq 1/(s \log^{\delta_2} s),$$
	\item[(B3')]
	$$\Delta_{\kP_s} \succeq 1/ \log s,$$
	\item[(B4')]
	$$P_\rho^G(Y_s\in B(s)^c) \preceq s^{-\gamma}\log^{\delta_3} s,$$
\end{itemize}
then
\[\psi_s(\rho) \preceq s^{-\gamma} \log^{\delta_4} s.\]
We observe that, in the choice of exponents and the quantities of interest, the condition (B1') is similar to (A3)(a) and the condition (B2') is related to (A3)(c). The remaining conditions, (B3') and (B4'), are more technical, and ensure the behaviour of the random walk on the subgraph suitably captures that on the entire graph. As (a simplification of) the main result of \cite{CS}, we state the following. We highlight that, although set-out here in terms of the continuous-time random walk, the conclusion is readily transferred to the discrete-time random walk by applying the argument used in the proof of \cite[Theorem 5.14]{BarHK}, for example.

\begin{theo} \cite[Theorem 1]{CS}
If $d\geq 1$ and $s\in (d,\min\{d+2,2d\})$, then the continuous-time simple random walk on the long-range percolation cluster described in the introduction satisfies the conditions (B1)--(B6) with the exponent $\gamma =d/(s-d)$. As a consequence, with probability one, the upper heat kernel bound
\[	\psi_T(\rho) \leq CT^{-d/(s-d)} \log ^{\delta} T\]
holds for all large $T$, where $C, \delta$ are deterministic constants.
\end{theo}

\section{Application to long-range percolation}\label{sec3}

In this section, we prepare the ground for deriving the lower heat kernel bounds for the long-range percolation model LRP($d$,$s$). Given a parameter $q \in [0,1]$, we slightly generalise the setting presented in the introduction around \eqref{pxy} by supposing: for $x,y \in \Z^d$, the edge between them appears with probability
\begin{equation*}
    q_{x,y} = \begin{cases}
    q& \textrm{if } |x-y|=1, \\
    1-\exp(-|x-y|^s)& \textrm{if } |x-y|>1,
    \end{cases}
\end{equation*}
independently of the state of other edges. In particular, when $q=1$, we assume nearest-neighbour edges are present. Note that we only allow $q<1$ in this section, where we discuss the verification of the conditions (A1)--(A3) for the more general model.

It was shown in \cite{AKN} that LRP($d$,$s$) admits an infinite cluster with probability $0$ or $1$, and  has at most one infinite cluster almost-surely. We will suppose that LRP($d$,$s$) percolates, i.e. it has an unique infinite cluster, which we denote by $G=(V,E)$; clearly this includes the $q=1$ setting, and indeed any $q$ above the nearest-neighbour percolation threshold for $\mathbb{Z}^d$. For each $n$, we let $G_n=(V_n,E_n)$ be the largest connected component of  $G \cap [-n,n]^d$, and assume the following.
\begin{itemize}
\item[(V)] There exists a universal constant $c>0$, such that for all $n\geq 1$,
    \begin{equation*}
        \pp(|V_n| \geq cn^d) \geq 1- \exp(-c(\log n)^2).
    \end{equation*}
  \end{itemize}
When $q=1$, the condition (V) is trivial, since  $V_n=[-n,n]^d$. Applying an estimate from \cite{B}, we can also check it in the case $s \in (d,2d)$, see Lemma \ref{vlem} below.

In the subsequent three subsections, we take the condition (V) as given, and proceed to check each of the assumptions (A1), (A2) and (A3). Thus we reduce the problem of obtaining heat kernel lower bounds to that of checking Benjamini-Schramm convergence and (V).

\subsection{Checking (A1) for long-range percolation under (V)}

The assumption (A1) is straightforward to handle for all cases simultaneously.

\begin{lem}\label{la1} For any $d\geq1$ and $s>d$, the random graph $G_n$  satisfies (A1).
\end{lem}
\begin{proof}
By applying the Cauchy-Schwarz inequality and (V), we have for some $c>0$,
\begin{eqnarray}  \label{eevn2}
\E \left( \left(\frac{|E_n|}{|V_n|} \right)^2 \right) & \leq & c^{-2} n^{-2d} \E[|E_n|^2] + \E[|E_n|^2 \I(|V_n|<cn^d)] \notag \\
&\leq& c^{-2} n^{-2d} \E[|E_n|^2] + (\E[|E_n|^4])^{1/2} (\pp[|V_n|<cn^d])^{1/2}.
\end{eqnarray}
In addition, writing $V_n^*=[-n,n]^d$,
\begin{eqnarray}\label{en3}
\E[|E_n|^2]  &\leq& \E \left( \left( \sum_{x,y \in V_n^*} \I(x \sim y) \right)^2 \right) = \sum_{x,y,u,v \in V_n^*} \pp[x \sim y, u \sim v] \notag\\
&\preceq&  \sum_{x,y,u,v \in V_n^*} (1+|x-y|^s)^{-1} (1+|u-v|^s)^{-1}  \preceq n^{2d}.
\end{eqnarray}
Similarly,
\begin{equation} \label{en4}
    \E[|E_n|^4]  \preceq n^{4d},
\end{equation}
Combining \eqref{eevn2}, \eqref{en3} and \eqref{en4} with  (V),  we obtain that $\E[(|E_n|/|V_n|)^2] \preceq 1$, as desired.
\end{proof}

\subsection{Checking (A2) for long-range percolation under (V)}

For assumption (A2), we can also deal with all cases simultaneously.

\begin{lem}\label{la2} For any $d\geq1$ and $s>d$, the random graph $G_n$ satisfies (A2).
\end{lem}
\begin{proof} Since $|E_n| \geq |V_n|-1$, it follows from (V) that
\begin{equation*}
    \pp[|E_n| \geq cn^d] \geq 1- \exp(-c(\log n)^2),
\end{equation*}
for some $c>0$. Therefore,
\begin{eqnarray} \label{pst1}
\E\left[\pi_n^\star(\varepsilon)^2\right] & =& \E \left[\sup_{|W|\leq \varepsilon |V_n|} \left(\frac{\sum_{x\in W} \textrm{deg}_{G_n}(x)}{2|E_n|}\right)^2\right] \notag \\
&\leq &  c^{-2}  \varepsilon + \exp(-c(\log n)^2)  + \pp \left(\sup_{|W|\leq \varepsilon |V_n|} \sum_{x\in W} \textrm{deg}_{G_n}(x) \geq \sqrt{\varepsilon} n^d\right).~~~~
\end{eqnarray}
Since $V_n \subseteq V_n^*:=[-n,n]^d$,  applying the union bound and Markov's inequality yields that
\begin{eqnarray} \label{sqep}
\pp \left(\sup_{|W|\leq \varepsilon |V_n|} \sum_{x\in W} \textrm{deg}_{G_n}(x) \geq \sqrt{\varepsilon} n^d\right) &\leq& \sum_{|W| \leq \varepsilon |V^*_n|} \pp \left(  \sum_{x\in W}\textrm{deg}_{G_n}(x) \geq \sqrt{\varepsilon} n^d\right) \notag \\
& \leq & \sum_{|W| \leq \varepsilon |V^*_n|} e^{-\sqrt{\varepsilon} n^d} \E \left[\exp \left( \sum_{x\in W}\textrm{deg}(x) \right) \right],
\end{eqnarray}
where $\deg (x) := \sum_{y \in \Z^d} \I(x \sim y)$. Uniformly in $n$ and $W$ with $|W| \leq \varepsilon |V^*_n|$, we have
\begin{eqnarray} \label{sdw}
 \E \left[\exp \left( \sum_{x\in W}\textrm{deg}(x) \right) \right] &=&  \E \left[\exp \left( \sum_{x\in W} \sum_{y \in \Z^d} \I(x\sim y) \right) \right]  \notag \\
 &=& \prod\limits_{y \in \Z^d  } \prod_{x \in W } \E \left(\exp \left(\I(x\sim y)  \right)\right) = \prod\limits_{y \in \Z^d   } \prod_{ x \in W } \left( 1-q_{x,y} + q_{x,y}e\right)  \notag \\
 & \leq & \prod\limits_{y \in \Z^d   } \prod_{ x \in W }  \exp(2q_{x,y}) = \exp \left(2\sum_{x\in W} \sum_{y \in \Z^d} \pp(x\sim y)\right)\nonumber\\
 & =& \exp \left(2\sum_{x\in W}  \E(\textrm{deg}(x)) \right)   \leq  \exp \left(C\varepsilon n^d\right).
\end{eqnarray}	
Combining \eqref{sqep} and \eqref{sdw}, and applying Stirling's formula, we obtain, for all $n$ large,
\begin{eqnarray} \label{pst3}
\pp \left(\sup_{|W|\leq \varepsilon |V_n|} \sum_{x\in W} \textrm{deg}_G(x) \geq \sqrt{\varepsilon} n^d\right) &\leq& \binom{|V^*_n|}{\varepsilon|V^*_n|} \exp\left(-\sqrt{\varepsilon} n^d +C\varepsilon n^d\right) \notag \\
& \leq & \exp \left(C I(\varepsilon) n^d -\sqrt{\varepsilon} n^d +C\varepsilon n^d \right),
\end{eqnarray}
where $I(\varepsilon) := -\varepsilon \log (\varepsilon) -(1-\varepsilon) \log (1-\varepsilon)$. Now, if $\varepsilon\in(0,\varepsilon_0)$ for suitably small $\varepsilon_0>0$, then $C (\varepsilon+I(\varepsilon))-\sqrt{\varepsilon}<0$, and we assume that this is the case. Putting \eqref{pst1} and \eqref{pst3} together, we thus deduce that, for all $n$ large, $\E[\pi_n^\star(\varepsilon)^2]\leq C\varepsilon$, as desired. Since $\E[\pi_n^\star(\varepsilon)^2]\leq 1$, the result is obviously true for $\varepsilon\in[\varepsilon_0,\infty)$, and so the proof is complete.
\end{proof}

\subsection{Checking (A3) for long-range percolation under (V)}

Recall that to verify the assumption (A3), we have to describe within the random graph $G_n$ a sequence of disjoint capacitors $\{(A_i,\Omega_i)\}_{i=1}^k$. Specifically, for (A3)(c), we need to establish an upper bound for the sum of the associated capacities in terms of $|E_n|$. Hence we should find an upper bound for the capacities and a lower bound for $|E_n|$. The condition (V) guarantees a lower bound on the number of edges. For the upper bound on the capacities, we will use the following observation. By the definition of the capacity and the Dirichlet form on the finite graph $G_n$, we have: for any function  $\varphi^{(i)}$ taking value $0$ outside $\Omega_i$ and value $1$ inside $A_i$,
\begin{equation*}
    \mathrm{cap}_{\Omega_i}(A_i) \leq 2 \kE_{G_n}(\varphi^{(i)}) \leq 2\kE(\varphi^{(i)}),
\end{equation*}
where $\kE$ is defined by setting
\[\kE(f) := \frac{1}{2} \sum \limits_{\substack{x,y \in \Z^d \\ x \sim y }} (f(x)-f(y))^2\]
for $f: \Z^d \rightarrow \R$. Since its definition does not involve summing over vertices in the largest percolation cluster in $[-n,n]^d$, computing with respect to $\kE$ is more convenient than doing so for $\kE_{G_n}$. In the next two results, Lemma \ref{lem:cov} and Proposition \ref{rvph}, we obtain some useful estimates on the covariance between $\kE(f)$ and $\kE(g)$ for suitable functions $f,g$, and the expectation of $\kE(f)$ for suitable $f$, respectively.

\begin{lem}  \label{lem:cov}
Let  $N \geq 1$ be an integer and $a,b \in \Z^d$ satisfy $|a-b| \geq 2N$. Suppose that $f_a, f_b$ are two bounded functions with $\supp(f_a) \subseteq \Omega_a$ and $\supp(f_b) \subseteq \Omega_b$, where $\Omega_a=a+[-N,N]^d$ and $\Omega_b=b+[-N,N]^d$. Then there exists an universal constant $C$ such that
\begin{equation*}
    \cov \left( \kE(f_a), \kE(f_b) \right) \leq \frac{C N^{4d} |f_a|^2 |f_b|^2}{(|a-b|-2N)^{2s}+1},
\end{equation*}
where we write $|f| := \sup_{x \in \Z^d} |f(x)|$ for $f:\Z^d \rightarrow \R$.
\end{lem}
\begin{proof}
Since $\supp(f_a) \subseteq \Omega_a$,
\begin{eqnarray*}
\kE(f_a) &=& \sum \limits_{\substack{\{x,y\}:\:x\in\Omega_a,\:y \in \Z^d,\\x \sim y }} (f_a(x)-f_a(y))^2 \\
&=& \sum \limits_{\substack{\{x,y\}:\:x\in\Omega_a,\:y \in \Omega_b \\ x \sim y }} (f_a(x)-f_a(y))^2 +  \sum \limits_{\substack{\{x,y\}:\:x\in\Omega_a,\:y \in \Omega_b^c \\ x \sim y }} (f_a(x)-f_a(y))^2 \notag \\
&=:& \kE_{a,1}+\kE_{a,2}.
\end{eqnarray*}
Similarly,
\begin{eqnarray*}
\kE(f_b) &=&  \sum \limits_{\substack{\{x,y\}:\:x\in\Omega_b,\: y \in \Omega_a \\ x \sim y }} (f_b(x)-f_b(y))^2 +   \sum \limits_{\substack{\{x,y\}:\:x\in\Omega_b,\: y \in \Omega_a^c \\ x \sim y }} (f_b(x)-f_b(y))^2 \\
&=:&  \kE_{b,1}+\kE_{b,2}.
\end{eqnarray*}
By construction, $\kE_{a,2}$ is independent of $\kE_{b,1}$ and $\kE_{b,2}$, and $\kE_{b,2}$ is independent of $\kE_{a,1}$ and $\kE_{a,2}$. Therefore,
\begin{eqnarray} \label{efab}
    \cov \left( \kE(f_a), \kE(f_b) \right) =     \cov \left( \kE_{a,1}, \kE_{b,1} \right) \leq \E \left( \kE_{a,1} \kE_{b,1} \right) \leq |f_a|^2 |f_b|^2 \E\left(e(\Omega_a,\Omega_b)^2 \right),
\end{eqnarray}
where
\[e(\Omega_a,\Omega_b): = \sum_{{\{x,y\}:\:x\in\Omega_b,\: y \in \Omega_a }} \I(x \sim y).\]
Writing $\lesssim$ for stochastic domination and $\mathrm{Bin}(m,p)$ for a binomial random variable with parameters $m \in \mathbb{N}$ and $p\in [0,1]$, we clearly have that
\[e(\Omega_a,\Omega_b) \lesssim \mathrm{Bin}(|\Omega_a| |\Omega_b|, p_{a,b}),\]
where
\[p_{a,b}: = \max_{x \in \Omega_a, y \in \Omega_b} q_{x,y}.\]
Since $|\Omega_a| |\Omega_b| \asymp N^{2d}$ and $p_{a,b} \asymp ((|a-b|-2N)^{s}+1)^{-1}$,
\begin{equation} \label{eoab}
    \E\left(e(\Omega_a,\Omega_b)^2 \right) \leq \E \left( \mathrm{Bin}(|\Omega_a| |\Omega_b|, p_{a,b})^2\right) \preceq \frac{N^{4d}}{(|a-b|-2N)^{2s}+1}.
\end{equation}
Combining \eqref{efab} and \eqref{eoab}, we arrive at the desired result.
\end{proof}

For the next step, we will need to be careful about the distinction between the stable, Gaussian, and critical settings. In particular, the capacity estimate we require differs between the three cases. Towards deriving this, we introduce a linear cut-off function, and estimate the expected value of its energy. For $N \geq M$, set
\[\varphi_{N,M}(x) = \begin{cases}
1, & \textrm{if } |x| \leq N-M ,\\
\frac{N-|x|}{M}, & \textrm{if } N-M<|x| \leq N, \\
0, & \textrm{if } |x| > N.
\end{cases}\]
Moreover, for a given sequence $(\beta_N)_{N\geq 1}$ in $[2,\infty)$, define $\varphi_N := \varphi_{N,M}$, where $M:=N/\beta_N$.  Note that the following result does not cover the part of the Gaussian regime corresponding to parameters $d=1$ and $s>2$, for which we deduce the lower heat kernel bound from the scaling limit of Theorem \ref{lrps}(b).

\begin{prop} \label{rvph}
(a) Fix $d\geq1$ and $s\in (d,\min\{d+2,2d\})$.
There exists a positive constant $c$ such that, for any sequence $(\beta_N)_{N\geq 1}$ in $[2,\infty)$ such that $\beta_N=o(N)$,  for all large $N$,
$$ cN^{2d-s}  \leq \E \left( \kE(\varphi_N) \right) \leq N^{2d-s} \beta_N/c.$$
(b) Fix $d\geq2$ and $s>d+2$. The statement of part (a) holds with the following bound:
$$cN^{d-2} \beta_N \leq \E \left( \kE(\varphi_N) \right) \leq N^{d-2} \beta_N/c.$$
(c)  Fix $d\geq2$ and $s=d+2$. The statement of part (a) holds with the following bound:
$$cN^{d-2}\beta_N \log (N/\beta_N) \leq \E \left( \kE(\varphi_N) \right) \leq N^{d-2} \beta_N \log(N/\beta_N)/c.$$
\end{prop}
\begin{proof}  (a)  By definition of $\varphi_N$,
\begin{eqnarray}
2\kE(\varphi_N) &=& \sum\limits_{\substack{|x| \leq N-M \\ |y|> N }} \I(x \sim y) + \sum\limits_{\substack{|x| \leq N-M \notag \\ N-M <|y| \leq  N }} \I(x \sim y)  \left(\frac{N-M-|y|}{M}\right)^2 \notag \\
&& + \sum\limits_{\substack{|x| >   N \\ N -M <|y| \leq  N }} \I(x \sim y)  \left(\frac{N-|y|}{M}\right)^2 +  \sum\limits_{N -M <|x|, |y|\leq  N } \I(x \sim y)\left(\frac{|x|-|y|}{M}\right)^2 \notag\\
&=:& S_1  + S_2 +S_3 + S_4,\nonumber
\end{eqnarray}
Using $q_{x,y} \asymp |x-y|^{-s}$ for $|x-y|\geq 1$, we have
\begin{eqnarray} \label{es1u}
\E(S_1) &\asymp& \sum\limits_{\substack{|x| \leq N-M \notag \\ |y|> N }} |x-y|^{-s}  = \sum \limits_{|x| \leq N -M } \sum_{k > N-|x|} k^{-s} \# \{y: |y| > N, |x-y|=k\}  \notag \\
& \asymp &  \sum \limits_{|x| \leq N -M } \sum_{k\in[N-|x|,2N-|x|]}  k^{d-1-s} \asymp  \sum \limits_{|x| \leq N -M } (N-|x|)^{d-s} \notag \\
& \asymp & \sum_{\ell=0}^{N-M} (N-\ell)^{d-s} \ell ^{d-1} \asymp \int_{1}^{N-M} (N-\ell)^{d-s} \ell^{d-1} d \ell \notag  \\
&\asymp& N^{2d-s} \int_{1/\beta_N}^{1-(1/N)} u^{d-s}(1-u)^{d-1} du,
\end{eqnarray}
where we used the fact that $\#\{y: |y-x|=k\} \asymp k^{d-1}$ in the second and third lines.
Observe that
\begin{equation} \label{itu}
\int_{1/\beta_N}^{1-(1/N)} u^{d-s}(1-u)^{d-1} du \asymp \begin{cases}
1 ,& \textrm{if } d+1-s >0, \\
\log(\beta_N),&  \textrm{if } d+1-s =0, \\
\beta_N^{s-d-1},& \textrm{if } d+1-s <0.
\end{cases}
\end{equation}
Combining \eqref{es1u} and \eqref{itu} yields that
\begin{equation} \label{ces1}
\E(S_1) \asymp \begin{cases}
N^{2d-s},& \textrm{if } d+1-s >0, \\
N^{2d-s}\log(\beta_N),& \textrm{if } d+1-s =0 ,\\
N^{2d-s}\beta_N^{s-d-1},&\textrm{if } d+1-s <0.
\end{cases}
\end{equation}
We next estimate $\E(S_4)$. Using the inequality $||x|-|y|| \leq |x-y|$, we have, for all $N-M \leq |x|,|y| \leq N$,
\begin{equation} \label{ks4}
\left(\frac{|x|-|y|}{M}\right)^2 \leq \frac{|x-y|^2}{M^2} \wedge 1.
\end{equation}
Therefore,
\begin{eqnarray} \label{us4}
\E(S_4) & = &\frac{1}{2} \sum_{N-M <|x|, |y| \leq N} q_{x,y} \left(\frac{|x|-|y|}{M}\right)^2  \notag \\
 & \preceq &\sum_{N-M <|x| \leq N} \left[\frac{1}{M^2} \sum_{y: |y-x| \leq M} |x-y|^{2-s} + \sum_{y: |y-x| > M} |x-y|^{-s}\right]\notag \\
 & \preceq & \sum_{N-M <|x| \leq N} \left[\frac{1}{M^2}\sum_{k=1}^{M} k^{2-s} k^{d-1} + \sum_{k>M} k^{-s} k^{d-1} \right] \notag \\
 & \preceq & \sum_{N-M <|x| \leq N} M^{d-s} \quad \preceq \quad N^{d-1} M^{d+1-s}=N^{2d-s}\beta_N^{s-d-1}.
\end{eqnarray}
Note that for the second line we used the fact that $\#\{y: |y-x|=k\} \asymp k^{d-1}$ and for the last line we used $\#\{x: N-M <|x| \leq N\}\asymp N^{d-1} M$. To estimate $\E(S_3)$, we observe that for $|x|>N$ and $N-M< |y| \leq N$,
\begin{equation} \label{ks3}
\left(\frac{N-|y|}{M}\right)^2 \leq \frac{|x-y|^2}{M^2} \wedge 1.
\end{equation}
Hence, using the same argument for as \eqref{us4}, with \eqref{ks3} playing the role of \eqref{ks4}, we can prove that
\begin{equation} \label{us3}
\E(S_3) \preceq N^{2d-s}\beta_N^{s-d-1}.
\end{equation}
The term $\E(S_2)$ can be also bounded in the same way as $\E(S_3)$ or $\E(S_4)$. Indeed, for $|x|\leq N-M$ and $N-M < |y|\leq N$,
\[\left(\frac{N-M-|y|}{M}\right)^2 \leq \frac{|x-y|^2}{M^2} \wedge 1.\]
Hence, using the argument that was used to obtain \eqref{us4} and \eqref{us3}, we find
\begin{equation} \label{us2}
\E(S_2) \preceq N^{2d-s}\beta_N^{s-d-1}.
\end{equation}
Using \eqref{ces1}, \eqref{us4}, \eqref{us3} and \eqref{us2}, we have  $\max \{\E(S_2), \E(S_3), \E(S_4)\} \preceq \E(S_1)$, and so
\[\E\left(\kE(\varphi_N) \right) \asymp \E(S_1),\]
which together with \eqref{ces1} implies that
\[ N^{2d-s} \preceq \E\left(\kE(\varphi_N) \right) \preceq N^{2d-s} \beta_N. \]
\noindent
(b)  Repeating the above argument yields
\[\E(S_1) \asymp N^{2d-s}\beta_N^{s-d-1}.\]
However, this is no longer the dominant term. Specifically, for $s>d+2$, we have that
\begin{eqnarray*}
\E(S_4) &\asymp& \sum\limits_{N -M <|x|, |y|\leq  N } |x-y|^{-s} \left(\frac{|x|-|y|}{M}\right)^2 \\
& \preceq &\sum\limits_{N -M <|x|\leq  N } \sum_{k\geq1}k^{-s} \left(\frac{k^2}{M^2}\wedge1\right)k^{d-1} \\
& \preceq &\sum\limits_{N -M <|x|\leq  N }\left( \sum_{k=1}^MM^{-2}k^{d+1-s} +\sum_{k>M}k^{d-1-s}\right)\\
&\preceq  & N^{d-1}M^{-1} = N^{d-2}\beta_N.
\end{eqnarray*}
Similar considerations yield a lower bound of the same form, i.e.\ $\E(S_4)\succeq N^{d-2}\beta_N$.  Moreover, again arguing similarly, we find that
$$\E(S_2)+\E(S_3)\preceq N^{d-2}\beta_N.$$
Therefore,
\[
\E\left( \kE(\varphi_N) \right) \asymp \E(S_4) \asymp N^{d-2}\beta_N.
\]
(c) The final case $d\geq 2$ and $s=d+2$.  The proof  is similar to that of (b). The additional log term appears since for $s=d+2$ it holds that $\sum_{k=1}^Mk^{d+1-s}\asymp \log (M)$.
\end{proof}

We are now ready to check (A3), and we start with the stable case.

\begin{lem}\label{a3s}
(a) Fix $d\geq 1$ and $s\in (d,\min\{d+2,2d\})$. For $\kappa>0$, LRP($d,s$) satisfies (A3) with $\delta_1=\frac{2d\kappa}{s-d}$, $\delta_2=\delta_3=\kappa$, $\delta_0$ arbitrarily large, and $\lambda(t)=\max\{1,\log(t)\}$.\\
(b) Fix $d\geq 1$ and $s\in (d,\min\{d+2,2d\})$. LRP($d,s$) satisfies (A3) with $\delta_1=\frac{2d}{s-d}$, $\delta_2=\delta_3=1$, $\delta_0$ arbitrarily large, and $\lambda(t)=\lambda_0$. Moreover, the constants $\alpha$ and $\lambda_0$ can be chosen so that, taking $C_\alpha$ as the constant of Theorem \ref{abth}, it holds that $1-C_\alpha\lambda_0^{-1/2}>0$.
\end{lem}
\begin{proof} (a) We will verify the condition with $n_0:=n_0(t)=t^{\Delta}$ suitably large $\Delta$.  To do so, first define
\begin{equation} \label{deon}
N:=t^{\tfrac{1}{s-d}} (\log t)^{\tfrac{2\kappa }{s-d}},
    \end{equation}
where $\kappa>0$ is some constant. For $n\geq n_0$, we then cover $[-n+N,n-N]^d$ by disjoint boxes of side-length $2N$, the intersections of which with $V_n$ we will denote by $(\Omega_i)_{i=1}^k$. Writing $x_i$ for the center of the box containing $\Omega_i$, so $\Omega_i=B_{x_i}(N) \cap V_n$ (where $B_x(r)$ is the $\ell_\infty$-ball of radius $r$ centred at $x$), we also introduce $A_i:=B_{x_i}(N-M) \cap V_n$, where $M:=N/(\log N)^\kappa$. Now, by construction, the sets $(\Omega_i)_{i=1}^k$ are disjoint, and moreover satisfy
\[\max_{i=1,\dots,k}|\Omega_i|\preceq N^d =t^{\tfrac{d}{s-d}} (\log t)^{\tfrac{2d\kappa }{s-d}}=t^\gamma (\log t)^{\delta_1},\]
where $\delta_1:=2d\kappa/(s-d)$. Thus (A3)(a) is satisfied for capacitors $(A_i,\Omega_i)$, $i=1,\dots,k$.

For verifying (A3)(b), define
\[    W = \{x:\:n-N\leq |x|\leq n\}\cup(\cup_{i=1}^kW_i),\]
where $W_i=\{x: N-M\leq |x-x_i|\leq N\}$, and observe that
\[   \sum_{i=1}^k\pi_n(A_i) \geq 1-\pi_n(W \cap V_n).\]
Now, for any $C$, $\varepsilon_n$ and suitably small constant $c$, we have
\begin{eqnarray*}
\pp\left(\pi_n(W \cap V_n) \geq C \varepsilon_n \right) &\leq&   \pp \left( \sum_{x \in W} \deg_{G_n}(x) \geq Cc \varepsilon_n n^d \right) + \pp(|E_n| \leq c n^d) \\
&\leq &  \pp \left( \sum_{x \in W} \deg(x) \geq Cc \varepsilon_n n^d \right) + \pp(|V_n| \leq c n^d+1)\\
&\leq&\exp(-Cc \varepsilon_n n^d) \E \left(\exp\left(\sum_{x \in W} \deg(x)\right) \right) + \exp(-c(\log n)^2)\\
&\leq &\exp\left(-Cc \varepsilon_n n^d +2|W|\mathbf{E}(\mathrm{deg}(\rho)) \right) + \exp(-c(\log n)^2),
\end{eqnarray*}
where for the  second and third inequalities we use $|V_n| \leq |E_n|+1$ and (V), and for the last one we follow the argument leading to \eqref{sdw}.  Furthermore, noting that the number of boxes satisfies $k\asymp\left(\frac{n}{N}\right)^d$,  we have
\[|W| = |\{x:\:n-N\leq |x|\leq n\}|+ \sum_{i=1}^k|W_i|\asymp n^{d-1}N+kN^{d-1}M\asymp n^d(\log N)^{-\kappa},\]
where  we have used that $N=o(n^{1/2})$ (for suitably large $\Delta$).  Hence, by taking $\varepsilon_n=(\log N)^{-\kappa}$ and $C$ sufficiently large, we find that
\[\pp\left(\pi_n(W \cap V_n) \geq  C(\log N)^{-\kappa}\right)\leq \exp(-n^d(\log N)^{-\kappa})+ \exp(-c(\log n)^2)\leq (\log t)^{-\delta}\]
for any $\delta>0$ (once $t$ is large). Reformulating this bound, we find that
\[\pp\left(\sum_{i=1}^k\pi_n(A_i)\geq 1-\alpha(\log t)^{-\kappa}\right)\geq 1-(\log t)^{-\delta},\]
which confirms (A3)(b) with $\delta_2=\kappa$.

Finally, to check (A3)(c), we start by defining a collection of functions $(\varphi^{(i)})_{i=1}^k$ by setting
$$\varphi^{(i)}(x)=\varphi_N(x-x_i).$$
Uniformly in $1 \leq i \leq k$,
\begin{equation*}
    \E \left( \kE(\varphi^{(i)})^2 \right) \leq \E \left( \left(\sum_{x \in \Omega_i} \deg(x) \right)^2 \right) \leq |\Omega_i|\E \left( \sum_{x \in \Omega_i} (\deg(x))^2 \right) \preceq |\Omega_i|^2 \preceq N^{2d}.
\end{equation*}
By Lemma \ref{lem:cov}, uniformly in all pairs  $i \neq j$,
\begin{equation*}
    \cov \left(\kE(\varphi^{(i)}), \kE(\varphi^{(j)}) \right) \preceq \frac{N^{4d}}{(|x_i-x_j|-2N)^{2s}+1}.
\end{equation*}
Therefore,
\begin{equation*}
\Var \left( \sum_{i=1}^k \kE(\varphi^{(i)}) \right) \preceq k N^{2d} + \sum_{1 \leq i \neq j \leq k}  \frac{N^{4d}}{(|x_i-x_j|-2N)^{2s}+1}.
\end{equation*}
In addition,
\begin{eqnarray*}
\lefteqn{\sum_{1 \leq i \neq j \leq k}  \frac{1}{(|x_i-x_j|-2N)^{2s}+1} = \sum_{1 \leq i \neq j \leq k}  \frac{(2N)^{-2s}}{ \left(|\tfrac{x_i-x_j}{2N}|-1\right)^{2s}+(2N)^{-2s}} }\\
&\asymp& \sum_{ x, y \in [-\tfrac{n}{2N}, \tfrac{n}{2N}]^d\cap \Z^d:\:x\neq y}  \frac{(2N)^{-2s}}{ \left(|x-y|-1\right)^{2s}+(2N)^{-2s}} \asymp \left(\frac{n}{N}\right)^d \asymp k,
\end{eqnarray*}
and so
\begin{equation} \label{varb}
\Var \left( \sum_{i=1}^k \kE(\varphi^{(i)}) \right) \preceq kN^{4d}.
\end{equation}
By Proposition \ref{rvph} (a), we also have that
\begin{equation} \label{expe}
kN^{2d-s} \preceq \E \left( \sum_{i=1}^k \kE(\varphi^{(i)}) \right)  \preceq k N^{2d-s} \beta_N \asymp n^dN^{d-s} (\log N)^{\kappa} \asymp n^d t^{-1} (\log t)^{-\kappa},
\end{equation}
where $\beta_N =N/M=(\log N)^{\kappa}$ and $N^{d-s} (\log N)^{\kappa} \asymp  t^{-1} (\log t)^{-\kappa}$ by  \eqref{deon}. Using  the variance bound \eqref{varb}, the lower bound for the expectation of \eqref{expe}, and Chebyshev's inequality, we thus obtain
\begin{eqnarray*}
\pp \left( \sum_{i=1}^k \kE(\varphi^{(i)}) \geq 2 \sum_{i=1}^k  \E[\kE(\varphi^{(i)})]\right) \preceq \frac{kN^{4d}}{(k N^{2d-s})^2} = \frac{N^{2s}}{k} \preceq k^{-1/2},
\end{eqnarray*}
provided that $ N^{4s} \preceq k$, or equivalently $N^{1+4s/d}\preceq n$ (which holds for $\Delta$ large). Since it is the case that $\supp(\varphi^{(i)})=\Omega_i$, $\varphi^{(i)}|_{A_i}=1$ and $0 \leq \varphi^{(i)} \leq 1$, we have $\mathrm{cap}_{\Omega_i}(A_i) \leq 2 \kE_{G_n}(\varphi^{(i)}) \leq 2\kE(\varphi^{(i)})$. Hence,
\[    \sum_{i=1}^k \mathrm{cap}_{\Omega_i}(A_i) \leq 2 \sum_{i=1}^k \kE(\varphi^{(i)}).\]
Furthermore, by (V),
\[\pp(|E_n| \leq c n^d) \leq \exp(-c(\log n)^2),\]
for $c$ a positive constant. Combining the last four displayed equations, we obtain that, for $\alpha$ suitably large,
\begin{eqnarray} \label{boca}
\lefteqn{\pp\left(\sum_{i=1}^k\mathrm{cap}_{\Omega_i}(A_i)> 2\alpha |E_n|t^{-1}(\log t)^{-\kappa}\right) }\notag\\
&\leq& \pp\left(\sum_{i=1}^k \kE(\varphi^{(i)})> \alpha |E_n|t^{-1}(\log t)^{-\kappa}\right) \notag \\
&\leq& \pp\left(\sum_{i=1}^k \kE(\varphi^{(i)})> \alpha cn^dt^{-1}(\log t)^{-\kappa}\right) + \pp(|E_n| \leq cn^d) \notag \\
&\leq& \pp \left( \sum_{i=1}^k \kE(\varphi^{(i)}) \geq 2 \sum_{i=1}^k  \E[\kE(\varphi^{(i)})]\right) + \exp(-c(\log n)^2) \notag\\
&\preceq& 2k^{-1/2}.
\end{eqnarray}
Hence we obtain (A3)(c) with $\delta_3=\kappa$.

\noindent
(b) Replacing $(\log t)^\kappa$ and $(\log N)^\kappa$ by $\lambda_0$ in the above argument, it follows that (A3) holds with the given constants. Furthermore, since the constant $\alpha$ that comes out of the argument does not depend on $\lambda_0$, by increasing the value of the latter quantity if needed, we can ensure that $1-C_\alpha\lambda_0^{-1/2}>0$.
\end{proof}

As for the Gaussian case, we check the following version of (A3). We underline that, although we include $d=2$ in the following result, we will handle the $d=1$ and $d=2$ cases separately in the proofs of our heat kernel estimates, using the quenched invariance principle that is known to hold throughout the Gaussian regime for those.

\begin{lem}\label{a3g}
(a) Fix $d\geq 2$ and $s> d+2$. For $\kappa\geq 1$, LRP($d,s$) satisfies (A3) with $\delta_1/d=\delta_2=\delta_3=\kappa$, $\delta_0$ arbitrarily large, and $\lambda(t)=\max\{1,\log(t)\}$.\\
(b) Fix $d\geq 2$ and $s> d+2$. LRP($d,s$) satisfies (A3) with $\delta_1/d=\delta_2=\delta_3=1$, $\delta_0$ arbitrarily large, and $\lambda(t)=\lambda_0$. Moreover, the constants $\alpha$ and $\lambda_0$ can be chosen so that, taking $C_\alpha$ as the constant of Theorem \ref{abth}, it holds that $1-C_\alpha\lambda_0^{-1}>0$.
\end{lem}
\begin{proof}
\noindent
(a) The argument is again similar to Lemma \ref{a3s}(a), but now we take
\[N:=t^{\tfrac{1}{2}} (\log t)^{\kappa}.\]
In this case, we then get that
\[\max_{i=1,\dots,k}|\Omega_i|\preceq N^d =t^{\tfrac{d}{2}} (\log t)^{\kappa d},\]
and so (A3)(a) holds with $\delta_1=\kappa d$. Proceeding as in the previous proof with $\beta_N:= (\log N)^{\kappa}$, we further have
\[\pp\left(\pi_n(W \cap V_n) \geq  C(\log N)^{-\kappa}\right) \leq (\log t)^{-\delta}\]
for any $\delta>0$ (once $t$ is large), which verifies (A3)(b) with $\delta_2=\kappa$.  By using similar arguments as for \eqref{boca}, noting that the variance bound \eqref{varb} holds for all $d,s$, and now with the help of Proposition \ref{rvph}(b), we can also prove that
\begin{eqnarray*}
\pp\left(\sum_{i=1}^k\mathrm{cap}_{\Omega_i}(A_i)> 2\alpha |E_n|t^{-1}(\log t)^{-\kappa}\right) \leq 2k^{-1/2},
\end{eqnarray*}
for $n\geq n_0(t)=t^{\Delta}$, for $\alpha, \Delta$ chosen suitably large.

\noindent
(b) Making appropriate adaptations to the proof of (a), the proof is similar to that of Lemma \ref{a3s}(b).
\end{proof}

Finally, in the critical case, we have the following. Note that in this case we do not provide a separate bound with $\lambda(t)$ constant as we do in the previous two lemmas, since the additional log term in Proposition \ref{rvph}(c) means that we are unable to avoid incorporating a log in the estimates somewhere.

\begin{lem}\label{a3c}
Fix $d\geq 2$ and $s=d+2$. For $\kappa>0$, LRP($d,s$) satisfies (A3) with $\delta_1/d-1/2=\delta_2=\delta_3=\kappa$, $\delta_0$ arbitrarily large, and $\lambda(t)=\max\{1,\log(t)\}$.
\end{lem}
\begin{proof} The argument is again similar to Lemma \ref{a3s}, but now we take
\[N:=t^{\tfrac{1}{2}} (\log t)^{\kappa+\frac{1}{2}},\]
and have to take into account the additional $\log$ term that appears in the critical case, i.e.\ apply Proposition \ref{rvph}(c), rather than Proposition \ref{rvph}(b).
\end{proof}

\section{Proof of Theorems \ref{lrpq} and \ref{lrpa}}\label{sec4}

\subsection{Proof of lower bounds}

Given the preparations of the previous section concerning (A1)--(A3), the main outstanding issue when it comes to the application of Theorem \ref{abth} is to check the assumption of Benjamini-Schramm convergence for the LRP($d$,$s$) model. Defining $G_n=(V_n,E_n)$, $n\geq 1$, as at the start of Section \ref{sec3}, we make precise the desired condition as follows.
\begin{itemize}
\item[(BS)] Let $\rho$ be the origin of $\Z^d$ and $\rho_n$ be a uniformly chosen vertex in $G_n$. Then the random graphs $(G_n,\rho_n)$ Benjamini-Schramm converge to $(G,\rho)$, conditioned that $\rho \in G$.
  \end{itemize}
In Section \ref{ss:ll}, we will explain how to verify (BS) in the non-nearest-neighbour setting when $s$ lies in the restricted range $(d,2d)$. Unfortunately, for the full range of parameters, $s>d$, we are only able to verify (BS) in the nearest-neighbour case, i.e.\ taking $q=1$ in the more general model of Section \ref{sec3}. (See Remark \ref{rem:lq} for further discussion of how our heat kernel estimates apply in general when both (V) and (BS) hold.)

\begin{lem}\label{bsclem} The LRP($d$,$s$) model with $q=1$ satisfies (BS).
\end{lem}
\begin{proof}
We write $B_G$ for balls in $G$ with respect to the graph distance, $B_{G_n}$ for balls in $G_n$ with respect to the graph distance, and $B_\infty$ for $\ell_\infty$-balls in $\mathbb{Z}^d$. We need to prove that for any rooted, finite graph $H$ and finite $r$,
\begin{equation} \label{BSc}
\lim_{n\rightarrow \infty}    \pp(B_{G_n}(\rho_n,r) =H) = \pp(B_G(\rho,r)=H).
\end{equation}
We have that
\begin{eqnarray*}
\mathbf{P}\left(B_{G_n}(\rho_n,r)=H\right)&\leq&\mathbf{P}\left(B_{G_n}(\rho_n,r)=H,\:B_{G}(\rho_n,r)\subseteq B_\infty(\rho_n,N)\subseteq B_\infty(\rho,n)\right)\\
&&+\mathbf{P}\left(B_{G}(\rho_n,r)\not\subseteq B_\infty(\rho_n,N)\right)+\mathbf{P}\left(B_\infty(\rho_n,N)\not\subseteq B_\infty(\rho,n)\right).
\end{eqnarray*}
Now, since $\rho_n$ is uniformly chosen on $V_n$,
\[\mathbf{P}\left(B_\infty(\rho_n,N)\not\subseteq B_\infty(\rho,n)\right)\leq \frac{cn^{d-1}N}{n^d}=\frac{cN}{n}\rightarrow 0,\]
as $n\rightarrow\infty$ (for each fixed $N$). Moreover, by the translation invariance of the model,
\[\mathbf{P}\left(B_{G}(\rho_n,r)\not\subseteq B_\infty(\rho_n,N)\right)= \mathbf{P}\left(B_{G}(\rho,r)\not\subseteq B_\infty(\rho,N)\right).\]
Since $G$ is $\mathbf{P}$-a.s.\ locally-finite when $s>d$, $B_G(\rho,r)$ is $\mathbf{P}$-a.s.\ a finite set, and so the above probability converges to 0 as $N\rightarrow\infty$. Finally, on the event $B_{G}(\rho_n,r)\subseteq B_\infty(\rho_n,N)\subseteq B_\infty(\rho,n)$, it holds that $B_{G_n}(\rho_n,r)=B_{G}(\rho_n,r)$, and so
\begin{eqnarray*}
\lefteqn{\mathbf{P}\left(B_{G_n}(\rho_n,r)=H,\:B_{G}(\rho_n,r)\subseteq B_\infty(\rho_n,N)\subseteq B_\infty(\rho,n)\right)}\\
&\leq& \mathbf{P}\left(B_G(\rho_n,r)=H\right)=\mathbf{P}\left(B_G(\rho,r)=H\right).
\end{eqnarray*}
In particular, it follows from what we have so far established that
\[\limsup_{n\rightarrow\infty}\mathbf{P}\left(B_{G_n}(\rho_n,r)=H\right)\leq \mathbf{P}\left(B_G(\rho,r)=H\right).\]
By noting that
\begin{eqnarray*}
\lefteqn{\mathbf{P}\left(B_{G_n}(\rho_n,r)=H\right)}\\
&\geq&\mathbf{P}\left(B_{G_n}(\rho_n,r)=H,\:B_{G}(\rho_n,r)\subseteq B_\infty(\rho_n,N)\subseteq B_\infty(\rho,n)\right)\\
&\geq & \mathbf{P}\left(B_G(\rho,r)=H\right)-\mathbf{P}\left(B_{G}(\rho,r)\not\subseteq B_\infty(\rho,N)\right)-\mathbf{P}\left(B_\infty(\rho_n,N)\not\subseteq B_\infty(\rho,n)\right),
\end{eqnarray*}
and applying the results of the preceding discussion, one may similarly conclude that
\[\liminf_{n\rightarrow\infty}\mathbf{P}\left(B_{G_n}(\rho_n,r)=H\right)\geq \mathbf{P}\left(B_G(\rho,r)=H\right),\]
which is enough to complete the proof of \eqref{BSc}.
\end{proof}

\begin{proof}[Proof of lower bounds of Theorem \ref{lrpq}] Combining Theorem \ref{lrps}(a) and Lemma \ref{biskuplem} gives the lower bound of \eqref{qs} in the case $s\in (d,d+1)$ with $\delta_1=0$. Similarly, combining Theorem \ref{lrps}(b) and Lemma \ref{biskuplem} gives the lower bounds of \eqref{qg} in the case $s>2d$ with $\delta_3=0$ and \eqref{qg2}.
In the remaining cases, we have from Lemmas \ref{la1}, \ref{la2}, \ref{a3s}(a), \ref{a3g}(a), \ref{a3c} that (A1), (A2) and (A3) hold for the $\delta_0,\delta_1,\delta_2,\delta_3$ and $\lambda(t)$ given by the latter three results. Taking $\kappa>2$ in Lemmas \ref{a3s}, \ref{a3g}, \ref{a3c} enables us to apply Corollary \ref{abthcor}(a) (and Lemma \ref{bsclem}) to derive quenched lower heat kernel bounds in each case, with the $\delta_i$s of \eqref{qs} and \eqref{qg} as in Remark \ref{deltaq}.
\end{proof}

\begin{proof}[Proof of lower bounds of Theorem \ref{lrpa}]
Again we appeal to Lemmas \ref{la1}, \ref{la2} to confirm that (A1) and (A2) hold in all three cases. In conjunction with Lemmas  \ref{a3s}(b) and \ref{a3g}(b), we have that (A3) holds in the sense required by Corollary \ref{abthcor}(b) in the stable and Gaussian cases. Putting this together with Lemma \ref{bsclem}, we thus obtain the lower bounds of \eqref{as} and \eqref{ag}. The lower bound of \eqref{ac} readily follows by applying Lemmas \ref{la1}, \ref{la2} and \ref{a3c}, with $\kappa>0$ chosen arbitrarily small, in conjunction with Theorem \ref{abth} (and Lemma \ref{bsclem}).
\end{proof}

\begin{rem} \label{rem:lq}
The lower bounds of Theorems \ref{lrpq} and \ref{lrpa} hold for the more general LRP($d$,$s$) model of Section \ref{sec3} (i.e.\ with $q\in[0,1]$) whenever the conditions (V) and (BS) are satisfied. Indeed, we proved in Section \ref{sec3} that under (V) the assumptions (A1)--(A3) are valid and hence, by arguments in the proof of Theorem \ref{abth},
\begin{equation*}
    \pp \left(p_{2t}^{G_n}(\rho_n, \rho_n) < t^{-\gamma}\lambda(t)^{-\delta} \right)  \preceq \lambda(t)^{-{\delta'}}
\end{equation*}
with $\gamma=\tfrac{d}{s-d}$ and $\lambda(t)=\log (t)$ and suitable $\delta,\delta'\geq 0$. Furthermore, the condition (BS)  assures the Benjamini-Schramm convergence of the  random graphs $G_n$, which in particular implies that
\begin{equation*}
    \left|\pp \left(p_{2t}^{G_n}(\rho_n, \rho_n) < t^{-\gamma}\lambda(t)^{-\delta_4} \right)  - \pp \left(p_{2t}^G(\rho,\rho) < 2t^{-\gamma}\lambda(t)^{-\delta_4}  \mid \rho \in G\right) \right|   \leq  \lambda(t)^{-{\delta_0}},
\end{equation*}
for all $n$ sufficiently large. The lower bounds of the heat kernel follow from the above two estimates; see Section \ref{ss:ll}, and Corollary \ref{corcorcor} in particular, for our application of this argument to the non-nearest-neighbour long-range percolation model of Section \ref{sec3} with $s\in (d,2d)$.
\end{rem}

\subsection{Proof of upper bounds}

\begin{proof}[Proof of upper bounds of Theorem \ref{lrpq}]
As noted in Remark \ref{deltaq}, the upper bound of \eqref{qs} follows from \cite[Theorem 1]{CS} (which did not require the assumption of nearest-neighbour edges being present), using the argument in the proof of \cite[Theorem 5.14]{BarHK} to transfer to discrete time. As for \eqref{qg2}, this is an immediate consequence of the general bound of \cite[Theorem 2.1]{BCG}, which implies that there exists a universal constant $C$ such that $p^G_{2t}(x,x)\leq Ct^{-1/2}$ for the simple random walk on any infinite connected graph $G$.

It remains to establish the upper bound of \eqref{qg}. To this end, we will first consider the continuous-time Markov process $(Z_t)_{t\geq 0}$, which has jump chain given by $X$, but the jump rate at site $x$ is equal to $\deg_G(x)$ (i.e.\ the holding time is exponential with this parameter).
The idea of the following proof comes from the unpublished version of \cite{BCKW}.
Note that the measure $m$ on $V$ placing mass $1$ on each vertex is invariant for $(Z_t)_{t\geq 0}$. We let
\[A_t:=\int_0^t\deg_G(Z_s)ds,\]
and define $(Y_t)_{t\geq 0}$ by setting $Y_t=Z_{A_t^{-1}}$, where $A_t^{-1}$ is the right continuous inverse of the non-decreasing additive functional $(A_t)_{t\geq 0}$; the process $(Y_t)_{t\geq 0}$ has the same jump chain as $Z$ (and $X$), but mean one exponential holding times. We claim that there exists a deterministic constant $c_1$ such that, for any realisation of $G$,
\begin{equation}\label{eq:noohr}
P^G_x(Z_t=x)\le c_1t^{-d/2}, \qquad \forall x\in \mathbb{Z}^d.
\end{equation}
Indeed, for the (constant speed) simple symmetric random walk on $\Z^d$, the Nash inequality
\[\Vert f\Vert_{L^2(m)}^{2+4/d}\le c_2\,{\mathcal E}^{\text{SRW}}(f,f)\,\Vert f\Vert_{L^1(m)}^{4/d}\]
holds for some constant $c_2=c_2(d)$, see for instance \cite[Lemma 3.13]{BarHK}.
Here, we have written $\|f\|_{L^p(m)}$ for the $L^p$-norm with respect to $m$, and $\mathcal{E}^{\text{SRW}}(f,f):=\frac12\sum_{x,y\in\mathbb{Z}^d:\:|x-y|=1}(f(x)-f(y))^2$. Since nearest-neighbour edges are present in $G$, it holds that ${\mathcal E}^{\text{SRW}}(f,f)\le{\mathcal E}(f,f)$, where $\mathcal{E}$ was defined at \eqref{edef}, and so the same inequality holds with ${\mathcal E}^{\text{SRW}}$ replaced by ${\mathcal E}$. By \cite[Theorem (2.1)]{CarKusStr}, we thus obtain \eqref{eq:noohr}.

We next estimate $P^G_x(Y_{2t}=x)$ by controlling the time change. Using the monotonicity of  $s\mapsto P^G_x(Y_{2s}=x)$, we get
\begin{eqnarray}
P^G_x(Y_{2t}=x)&\le &  \frac 1tE^G_{x}\left(\,\int_t^{2t}\mathbb{I}(Y_s=x) ds\right)\nonumber\\
&=&\frac 1tE^G_{x}\left(\,\int_{A_t^{-1}}^{A_{2t}^{-1}}\mathbb{I}(Z_u=x)A_u'du\right)\nonumber\\
&\le&\frac {\deg_G(x)}t
\int_0^{2t} P^G_x\left(Z_u=x, t\le A_u\le 2t\right)du,\label{eq:qwehsr}
\end{eqnarray}
where we first changed variables using $s:=A_u$,  and then used that the derivative $A_u'$ satisfies $A_u'=\deg_G(Z_u)=\deg_G(x)$ on the event $Z_u=x$. In the last inequality, we also used the fact that $A_{2t}\ge 2t$, which holds because $\deg_G(x)\ge 2d\ge 1$ (since all nearest-neighbour edges are present).

Now, let $D_t:=\sup_{0\le s\le t}|Y_s|$. By \cite[Lemma 4.1]{CS}, there exist $c_3,c_4,c_5>0$ such that, for any $T,\lambda>0$
and any $p>(s-d)^{-1}$, $r<s-d$,
\begin{equation}\label{dtbound}
{\mathbf P}\left(\{ P^G_\rho\left(D_t\ge c_3t^{p+1}\right)>c_4t^{-\lambda}\}\right)\le c_5t^{\lambda+1-pr}.
\end{equation}
Hence, taking $p$ large enough so that $\sum_{t=1}^\infty t^{d/2+1-pr}<\infty$, by applying the Borel-Cantelli lemma, one can deduce that, $\mathbf{P}$-a.s., for all large $t\in\mathbb{N}$,
\begin{equation}\label{etbound}
P^G_\rho(E_t^c)\le c_{4} t^{-d/2},
\end{equation}
where $E_t:=\{D_t<c_3t^{p+1}\}$. Since
\[D_{2t}=\sup_{0\leq s\leq 2t}|Y_s|=\sup_{0\leq s\leq 2t}|Z_{A^{-1}_s}|=\sup_{0\leq s\leq A^{-1}_{2t}}|Z_s|,\]
it further holds that $E_{2t}\cap\{A_u\leq {2t}\}\subseteq F_{u,t}$, where $F_{u,t}:=\{\sup_{0\leq s\leq u}|Z_s|\leq c_3(2t)^{p+1}\}$. In particular, applying \eqref{etbound}, we deduce from \eqref{eq:qwehsr} that, $\mathbf{P}$-a.s., for all large $t\in\mathbb{N}$,
\begin{equation}\label{inter1}
\frac{P^G_x(Y_{2t}=x)}{\deg_G(x)}\le 2c_4(2t)^{-d/2}+\frac {1}{t}
\int_0^{2t} P^G_x\left(Z_u=x, t\le A_u\le 2t, F_{u,t}\right)du.
\end{equation}
Using the Markov property, we moreover have
\begin{eqnarray*}
\lefteqn{P^G_x(Z_u=x, t\le A_u\le 2t, F_{u,t})}\\
&\le&\sum_{y\in\mathbb{Z}^d}P^G_x(Z_{u/2}=y, Z_u=x, t/2\le A_{u/2}\le 2t, F_{u,t})\\
&&~~~~~~~+\sum_{y\in\mathbb{Z}^d}P^G_x(Z_{u/2}=y, Z_u=x, t/2\le A_u-A_{u/2}\le 2t, F_{u,t})\\
&\leq&
\sum_{y\in\mathbb{Z}^d} E^G_x\left(\mathbb{I}(Z_{u/2}=y,t/2\le A_{u/2}\le 2t, F_{u/2,t})P^G_y\left(Z_{u/2}=x\right)\right)\\
&&~~~~~~~+\sum_{y\in\mathbb{Z}^d}E^G_x\left(\mathbb{I}(Z_{u/2}=y)P^G_y\left(Z_{u/2}=x, t/2\le A_{u/2}\le 2t, F_{u/2,t}\right)\right).
\end{eqnarray*}
Noting that $P^G_y(Z_{u/2}=x)\le c_6 u^{-d/2}$, which is due to \eqref{eq:noohr} and the Cauchy-Schwarz inequality, and
\[    P^G_y\left(Z_{u/2}=x, t/2\le A_{u/2}\le 2t, F_{u/2,t}\right)=P^G_x\left(Z_{u/2}=y, t/2\le A_{u/2}\le 2t, F_{u/2,t}\right),\]
which is due to the symmetry of $(Z_t)_{t\geq 0}$, we obtain
\begin{eqnarray*}
\lefteqn{P^G_x(Z_u=x, t\le A_u\le 2t, F_{u,t})}\\
&\leq& 2\sum_{y\in\mathbb{Z}^d}P^G_x\left(Z_{u/2}=y\right)P^G_x\left(Z_{u/2}=y, t/2\le A_{u/2}\le 2t, F_{u/2,t}\right)\\
&\leq& c_7u^{-d/2}P^G_x\left(t/2\le A_{u/2}\le 2t, F_{u/2,t}\right).
\end{eqnarray*}
Plugging this into \eqref{inter1}, we have, $\mathbf{P}$-a.s., for all large $t\in\mathbb{N}$,
\begin{equation}\label{eq:noeo2n}
\frac{P^G_x(Y_{2t}=x)}{\deg_G(x)}\le c_8t^{-d/2}+\frac {c_9}{t}
\int_0^{2t}u^{-d/2}P^G_x\left(t/2\le A_{u/2}\le 2t, F_{u/2,t}\right) du.
\end{equation}
Note, on $F_{u/2,t}$, it holds that
\[A_{u/2}=\int_0^{u/2}\mathrm{deg}_G(Z_s)ds\leq \frac{u}{2} \max_{|x|\leq c_3(2t)^{p+1}}\mathrm{deg}_G(x),\]
and so
\begin{eqnarray}
P^G_x\left(t/2\le A_{u/2}\le 2t, F_{u/2,t}\right) &\leq&\left(\frac{u}{t}\right)^{(d-1)/2} E^G_x\left(\left(\frac{A_{u/2}}{u/2}\right)^{(d-1)/2}\mathbb{I}(F_{u/2,t})\right)\label{aaa1}\\
&\leq& \left(\frac{u}{t}\right)^{(d-1)/2}  \max_{|x|\leq c_3(2t)^{p+1}}\mathrm{deg}_G(x)^{(d-1)/2}.\nonumber
\end{eqnarray}
In particular, together with \eqref{eq:noeo2n}, this implies, $\mathbf{P}$-a.s., for all large $t\in\mathbb{N}$,
\[\frac{P^G_x(Y_{2t}=x)}{\deg_G(x)}\leq c_{10} t^{-d/2}\max_{|x|\leq c_3(2t)^{p+1}}\mathrm{deg}_G(x)^{(d-1)/2}.\]
For bounding the max term, we note that
\[{\mathbf P}\left(\max_{|x|\le c_{3}(2t)^{p+1}}\deg_G(x)\ge c_{11}\log t\right)\le c_{3}(2t)^{d(p+1)}e^{-c_{11}\log t}
{\mathbf E}[e^{\deg_G (\rho)}]\le c_{12}t^{d(p+1)-c_{11}},\]
which, by taking $c_{11}>d(p+1)+1$, is summable over $t\in\mathbb{N}$. Consequently, on applying the Borel-Cantelli lemma, we obtain an estimate of the desired form for $(Y_t)_{t\geq 0}$.

To complete the proof, we need to transfer the estimate to discrete time. Note that we can write $Y_t=X_{T_t}$, where $(T_t)_{t\geq 0}$ is a unit rate Poisson process on $[0,\infty)$, independent of $X$. Hence we have that
\begin{eqnarray*}
P^G_x(Y_t=x)&=&P^G_x\left(X_{T_t}=x\right)\\
&=&\sum_{s\geq 0}P^G_x\left(X_{2s}=x\right)P^G_x(T_t=2s)\\
&\geq &\min_{s\in[t/4,t]}P^G_x\left(X_{2s}=x\right)P^G_x(T_t\in[t/2,2t])\\
&\geq &P^G_x(X_{2t}=x)\left(1-P^G_t(|T_t-t|>t/2)\right).
\end{eqnarray*}
By Chebyshev's inequality, it holds that $P^G_t(|T_t-t|>t/2)\leq \frac{4}{t^2}\mathrm{Var}_x^G(T_t)=\frac{4}{t}$. Consequently, for $t\geq 8$, it holds that $P^G_x(X_{2t}=x)\leq 2P^G_x(Y_t=x)$, and so the result follows from the continuous-time estimate.
\end{proof}

\begin{rem}\label{nologrem}
By Jensen's inequality and Fubini's theorem, we have that
\begin{equation}\label{aaa2}
E^G_x\left(\left(\frac{A_{u/2}}{u/2}\right)^{(d-1)/2}\right)\leq \frac{2}{u}\int_0^{u/2} E^G_x\left(\mathrm{deg}_G(Z_s)^{(d-1)/2}\right)ds.
\end{equation}
Hence if we could prove a bound of the form $E^G_x(\mathrm{deg}_G(Z_s)^{(d-1)/2})\leq C_G$ for all $s\geq 0$, where $C_G$ is a random constant depending only on the environment $G$, then we would obtain the quenched upper bound without a logarithm.
\end{rem}

\begin{proof}[Proof of upper bounds of Theorem \ref{lrpa}]
Similarly to \eqref{qs}, the upper bound of \eqref{as} follows from \cite[Theorem 1]{CS}. Indeed, it is proved there that there exist deterministic constants $c,\delta$ such that, $\mathbf{P}$-a.s.,
\[p^G_{2t}(\rho,\rho)\leq c t^{-\frac{d}{s-d}}\left(\log t\right)^\delta\]
holds for $t\geq T$, where $T$ is a random variable that satisfies: for any $\eta>0$, there exists a constant $C$ such that \[\mathbf{P}\left(T> t\right)\leq Ct^{-\eta}.\]
(Again, we highlight that, although the general bound of \cite{CS} is given for the continuous-time random walk, this is readily transferred to discrete time by applying the argument used in the proof of \cite[Theorem 5.14]{BarHK}.) Hence taking $\eta\geq d/(s-d)$ yields
\[\mathbf{E}\left(p^G_{2t}(\rho,\rho)\right)\leq c t^{-\frac{d}{s-d}}\left(\log t\right)^\delta + \mathbf{P}\left(T> t\right)\leq ct^{-\frac{d}{s-d}}\left(\log t\right)^\delta.\]

The proof of \eqref{ac} and \eqref{ag} can be obtained using the estimates in the quenched cases. Indeed,
for $d=1$, one just takes the expectation under $\mathbf E$ of \eqref{qg2}, recalling from the proof of the latter result that the constant in the upper bound is deterministic, and the bound holds for all $t\in\mathbb{N}$. For $d\ge 2$, we return to \eqref{eq:noeo2n}, replacing the first term in the upper bound by the probability that it is bounding:
\[\frac{P^G_x(Y_{2t}=x)}{\deg_G(x)}\le 2P_x^G(E_{2t}^c) +\frac {c_9}{t}
\int_0^{2t}u^{-d/2}P^G_x\left(t/2\le A_{u/2}\le 2t, F_{u/2,t}\right) du.\]
For the expectation of the first term, we have from \eqref{dtbound} with $\lambda=d/2$ and $p$ chosen suitably large that
\[\mathbf{E}\left(P_x^G(E_{2t}^c)\right)\leq c_4t^{-d/2} +c_5t^{d/2+1-pr}\leq c_{13}t^{-d/2}.\]
For the expectation of the second term, we apply \eqref{aaa1} and \eqref{aaa2} to deduce that
\[\mathbf{E}\left(P^G_x\left(t/2\le A_{u/2}\le 2t, F_{u/2,t}\right) \right)\leq  \left(\frac{u}{t}\right)^{(d-1)/s}\frac{2}{u}\int_0^{u/2}\mathbf{E}\left( E^G_x\left(\mathrm{deg}_G(Z_s)^{(d-1)/2}\right)\right)ds.\]
To obtain the desired bound in the continuous-time setting, it follows that it is enough to prove that there exists a constant $C$, independent of $s$, such that
$\mathbf{E}( E^G_x(\mathrm{deg}_G(Z_s)^{(d-1)/2}))\leq C$. To prove this, note that the environment process $(G_{Z_s})_{s\geq 0}$,  where $G_x$ is the graph $G$ translated by $-x$, is invariant and reversible with respect to the measure $m$ (as introduced in the previous proof) under the annealed measure, see, for instance, \cite[Section 4.1]{CS}. Thus, setting $f(G):=\deg_G (\omega(\rho))^{(d-1)/2}$, we have
\begin{eqnarray*}
\mathbf{E}\left( E^G_x\left(\mathrm{deg}_G(Z_s)^{(d-1)/2}\right)\right)&=&\mathbf{E}\left( E^G_x\left(f(G_{Z_s})^{(d-1)/2}\right)\right)\\
&=&\mathbf{E}\left( E^G_x\left(f(G_{Z_0})^{(d-1)/2}\right)\right)\\
&=& \mathbf{E}\left(\mathrm{deg}_G(x)^{(d-1)/2}\right),
\end{eqnarray*}
and we further have that the right hand side is finite when $s>d$. In particular, this establishes that
\[\mathbf{E}\left(\frac{P^G_x(Y_{2t}=x)}{\deg_G(x)}\right)\leq ct^{-d/2}.\]
We can transfer this result to the discrete-time process $(X_t)_{t\geq 0}$ exactly as in the quenched case.
\end{proof}

\section{Open questions}\label{oqsec}

Now we have completed the proofs of our main results, we collate a number of issues left open by the present work. (Some of these are discussed in more detail elsewhere.)
\begin{enumerate}
  \item In the quenched results of \eqref{qs} and \eqref{qg}, and the annealed results of \eqref{as} and \eqref{ac}, it is natural to optimise the log exponents. For the annealed bounds, one would expect a log term in the case $d\geq2$, $s=d+2$ only. One might conjecture that this is also the case for the quenched bounds.
  \item As in \cite{CS,CS2}, one might seek to derive similar heat kernel bounds to ours when the assumption that nearest-neighbour bonds are present is dropped. At least for the lower bounds, we have reduced the problem to checking the conditions (V) and (BS) (recall Remark \ref{rem:lq}), and verified these in the case $s\in (d,2d)$ (see Section \ref{ss:ll} below). Is it also possible to check (V) and (BS) in the case $s\geq 2d$?
  \item Our results support the extension of the Gaussian regime of Theorem \ref{lrps}(b) to $d\geq 1$, $s> \min\{d+2,2d\}$, and the extension of the stable scaling regime of Theorem \ref{lrps}(a) to $d\geq 1$, $s\in (d,\min\{d+2,2d\})$. Can this be proved? Some discussion of the latter case is provided in \cite[Section 3]{CS2}.
  \item In what sense is it possible to determine the walk dimension of $(X_t)_{t\geq0}$, that is, the exponent governing the space-time scaling of this process? A related problem is to establish bounds that satisfactorily describe the off-diagonal decay of the heat kernel, for which the techniques of the current article are insufficient. (As noted in the introduction, for nearest-neighbour bond percolation, quenched and annealed off-diagonal Gaussian heat kernel estimates are established in \cite{MTB}.)
  \item All questions remain open in the case $d=1$, $s=2$. What can be said here?
  \item Throughout this paper, we consider unweighted random graphs partly because Proposition \ref{kpr} (which is \cite[Theorem 3.7]{L}) is stated in this setting. It would be interesting to extend our results to weighted random graphs, including those arising in
      random conductance models. Such an extension would potentially be applicable to the model of \cite{XCKW}. In particular, the latter paper established  heat kernel estimates for long-range random conductance models on integer lattices when the conductance between $x$ and $y$ is given by $w_{x,y}/|x-y|^{d+\alpha}$, where $\{w_{x,y}=w_{y,x}\ge 0: x,y \in \Z^d\}$ are independent and satisfy some moment condition. Is it possible to extend our approach cover this random conductance model when the $\{w_{x,y}\}$ have a translation-invariant distribution?
\end{enumerate}

\section{Appendix}\label{secA}

We finish with a miscellany of results related to the heat kernel estimation techniques and long-range percolation model of this article. Lemma \ref{biskuplem} in particular is required for our lower heat kernel estimates.

\subsection{Heat kernel lower bounds on $\mathbb{Z}^d$}\label{zdsec}

In this section, we explain how to check the assumptions of Section \ref{hkesec} for the graph $\mathbb{Z}^d$. For the statement of the next result, we write $B_n$ for the $\ell_\infty$-ball of radius $n$, centred at 0, and $R$ for the effective resistance on the integer lattice (see \cite[Chapter 2]{BarHK}, for example).

\begin{lem} For $0\leq m<n$,
 \[R\left(B_m,B_n^c\right)\geq \left\{
                           \begin{array}{ll}
                             c_1(n-m), & \hbox{for $d=1$;} \\
                             c_2(\log(n)-\log(m)), & \hbox{for $d=2$;} \\
                             c_d(m^{2-d}-n^{2-d}), & \hbox{for $d\geq3$.}
                           \end{array}
                         \right. .\]
\end{lem}
\begin{proof} By applying the Nash-Williams inequality (see \cite[Proposition 9.15]{LPW}, for example), we have the following:
 \[R\left(B_m,B_n^c\right)\geq \sum_{l=m}^{n-1}\frac{1}{l^{d-1}},\]
 from which the result readily follows.
\end{proof}

Now, tile $B_n$, $n\geq t^{1/2}$, by boxes $(\Omega_i)$ that are each translations of $B_{t^{1/2}}$. Let $A_i$ be the central part of $\Omega_i$, as given by a translation of $B_{(1-\lambda^{-1})t^{1/2}}$. It is then the case that
\[\sum_i\mathrm{cap}_{\Omega_i}(A_i)\preceq\left\{
                           \begin{array}{ll}
                             \frac{n}{t^{1/2}}\times t^{-1/2}\lambda=C_\lambda nt^{-1}, & \hbox{for $d=1$;} \\
                             \left(\frac{n}{t^{1/2}}\right)^2\times \log(1/(1-\lambda^{-1}))^{-1}=C_\lambda n^2t^{-1}, & \hbox{for $d=2$;} \\
                             \left(\frac{n}{t^{1/2}}\right)^d\times\frac{\left(t^{1/2}\right)^{d-2}}{(1-\lambda^{-1})^{2-d}-1} =C_\lambda n^dt^{-1}, & \hbox{for $d\geq3$.}
                           \end{array}
                         \right. .\]
This is enough to check (A3) in this setting. The remaining assumptions are straightforward to check.

\subsection{Quenched lower bound from simple random walk scaling limit}

The following bound is adapted from \cite[Lemma 5.1]{Biskup}, and implies that a scaling limit for a random walk on a random graph of an appropriate form immediately yields a quenched heat kernel lower bound.

\begin{lem}\label{biskuplem} Let $X$ be a simple random walk on a connected, locally-finite graph $G=(V,E)$, started at root vertex $\rho$, and $p_t^G(x,y)$ be its heat kernel (with respect to the measure $\tilde{\pi}(x)=\mathrm{deg}_G(x)$). Let $d_G$ be a metric on $V$, and suppose that: for some constants $d_w,d_f\in(0,\infty)$, the laws of
\begin{equation}\label{pb}
\left(n^{-\frac{1}{d_w}}d_G(\rho,X_{nt})\right)_{t\geq 0},\qquad n\geq 1,
\end{equation}
form a tight sequence in $L^1([0,1])$, and also
\begin{equation}\label{mb}
\sup_{n\geq 1} n^{-{d_f}}\tilde{\pi}\left(\{x:\:d_G(\rho,x)\leq n\}\right)<\infty.
\end{equation}
It is then the case that there exists a constant $c>0$ such that: for all $t\geq 1$,
\[p^G_{2t}(\rho,\rho)\geq ct^{-\frac{d_f}{d_w}}.\]
\end{lem}
\begin{proof} We have that
\begin{eqnarray*}
p^G_{2t}(\rho,\rho)&=&\sum_{x\in V}p_t^G(\rho,x)p_t^G(x,\rho)\tilde{\pi}(x)\\
&\geq&\sum_{x\in V:\:d_G(\rho,x)\leq n}p_t^G(\rho,x)^2\tilde{\pi}(x)\\
&\geq& \frac{P^G_\rho\left(d_G(\rho,X_t)\leq C n\right)^2}{\tilde{\pi}\left(\{x\in V:\:d_G(\rho,x)\leq Cn\}\right)}\\
&\geq & c(Cn)^{-{d_f}}P^G_\rho\left(d_G(\rho,X_t)\leq Cn\right)^2
\end{eqnarray*}
where for the first inequality we use the symmetry of the heat kernel for the second we apply the Cauchy-Schwarz inequality, and for the third we appeal to \eqref{mb}. Now, by applying the monotonicity of the on-diagonal part of the heat kernel, it follows that
\begin{eqnarray*}
p^G_{2t}(\rho,\rho)&\geq&\left(t^{-1}\sum_{s=t}^{2t-1}p^G_{2s}(\rho,\rho)^\frac{1}{2}\right)^2\\
&\geq &c(Cn)^{-{d_f}}\left(t^{-1}\sum_{s=t}^{2t-1}P^G_\rho\left(d_G(\rho,X_s)\leq Cn\right)\right)^2\\
&=&c(Cn)^{-{d_f}}\left(E^G_\rho\left(t^{-1}\sum_{s=t}^{2t-1}\mathbb{I}_{\{d_G(\rho,X_s)\leq Cn\}}\right)\right)^2.
\end{eqnarray*}
Setting $n=t^{\frac{1}{d_w}}$, this yields
\[p^G_{2t}(\rho,\rho)\geq cC^{-d_f}t^{-\frac{d_f}{d_w}} \left(E^G_\rho\left(t^{-1}\sum_{s=t}^{2t-1}\mathbb{I}_{\{d_G(\rho,X_s)\leq Ct^{\frac{1}{d_w}}\}}\right)\right)^2.\]
Finally, given \eqref{pb}, the Kolmogorov-Riesz compactness theorem (see \cite[Theorem 5]{KR}, for example) implies that, by taking a suitably large value of $C$,
\[E^G_\rho\left(t^{-1}\sum_{s=t}^{2t-1}\mathbb{I}_{\{d_G(\rho,X_s)\leq Ct^{\frac{1}{d_w}}\}}\right)=1-E^G_\rho\left(t^{-1}\sum_{s=t}^{2t-1}\mathbb{I}_{\{t^{-\frac{1}{d_w}}d_G(\rho,X_s)> C\}}\right)\geq \frac{1}{2},\]
uniformly in $t$. In conjunction with the previous bound, this completes the proof.
\end{proof}

\begin{rem}
In examples, $d_w$ will typically represent the so-called walk dimension of $X$, which is the exponent governing the space-time scaling of the random walk (with respect to the metric $d_G$). This is also sometimes called the escape time exponent. Moreover, $d_f$ will be the volume growth exponent (again, with respect to the metric $d_G$). In the case when $d_G$ is the usual shortest path graph distance (on a suitably regular graph), discussion of the possible values of $d_w$ and $d_f$ appears in \cite{BarEsc}.
\end{rem}

\begin{rem} If in place of condition \eqref{pb} one had that the laws of
\[n^{-\frac{1}{d_w}}d_G(\rho,X_{n}),\qquad n\geq 1,\]
form a tight sequence, then one would be able to deduce the same result by an easier proof. In particular, the integration over time would not be necessary. Whilst this would be enough for us in the Gaussian case, we need the above $L^1$ version to deal with the weaker convergence statement that is known to hold in the stable case.
\end{rem}

\subsection{Quenched invariance principle in one-dimension via resistance scaling}\label{qipd1}

In this section, via the resistance scaling techniques of \cite{Croy} (see also \cite{CHK}), we establish a quenched invariance principle for simple random walk on LRP($d$,$s$) for $d=1$ and $s>2$ (cf.\ Theorem \ref{lrps}(b)), which, in conjunction with Lemma \ref{biskuplem}, gives an on-diagonal lower bound for the heat kernel of the model in question. We assume nearest-neighbour edges are present, i.e.\ $q=1$. Our proof gives an alternative viewpoint to the arguments of \cite{CS2}, which used a martingale approach, and \cite{ZZ}, which applied the corrector method. (Note that it was also the case in both \cite{CS2} and \cite{ZZ} that nearest-neighbour edges were assumed to be present, which ensures percolation occurs; see \cite{Bperc} for discussion of percolation for long-range percolation with $d=1$.) The particular form of our proof is closely related to that used to understand the scaling of the Mott random walk in \cite{CFJ}. Since it is not an original result, we are brief with the details.

\begin{prop} If $d=1$ and $s>2$, then for $\mathbf{P}$-a.e.\ realisation of LRP($d,s$), the law of
\[\left(n^{-\frac{1}{2}}X_{nt}\right)_{t\geq 0}\]
on $C([0,\infty))$ converges weakly to that of $(B_{\sigma^2 t})_{t\geq 0}$, where $(B_{t})_{t\geq 0}$ is standard Brownian motion on $\mathbb{R}^d$, and $\sigma^2\in(0,\infty)$ is a deterministic constant.
\end{prop}
\begin{proof} Let $\tilde{\pi}$ be the measure on $\mathbb{Z}$ given by $\tilde{\pi}(\{x\})=\mathrm{deg}_G(x)$. Since $C:=\mathbf{E}(\tilde{\pi}(\{0\}))\in(0,\infty)$, we readily deduce from the ergodic theorem that, $\mathbf{P}$-a.s.,
\begin{equation}\label{claims1}
n^{-1}\tilde{\pi}\left(\{an,\dots,bn\}\right)\rightarrow C(b-a),\qquad \forall a,b\in\mathbb{R},\:a<b.
\end{equation}

Writing $R$ for the effective resistance on $\mathbb{Z}$, we further claim that there exists a deterministic constant $R_\infty\in(0,\infty)$ such that, $\mathbf{P}$-a.s.,
\begin{equation}\label{claims2}
\left(n^{-1}R(xn,yn)\right)_{x,y\in\mathbb{R}}\rightarrow \left(R_{\infty}|x-y|\right)_{x,y\in\mathbb{R}},
\end{equation}
uniformly on compacts. To check this, we first apply the triangle inequality for the effective resistance and Kingman's subadditive ergodic theorem to deduce that, $\mathbf{P}$-a.s.\
\[\left(n^{-1}R(0,xn)\right)_{x\in\mathbb{R}}\rightarrow \left(R_{\infty}|x|\right)_{x\in\mathbb{R}},\]
uniformly on compacts, where $R_\infty:=\inf_{n\geq 1} n^{-1}\mathbf{E}R(0,n)$. We note that $R_\infty\leq 1$, because the presence of nearest-neighbour edges ensures that $R(0,n)\leq n$. Now, from \cite[Lemma 10.1]{CS2}, we have that the cut-points of the underlying graph are dense on the appropriate scale, where we say that $x$ is a cut-point for $G$ if $\{x,x+1\}$ is the only edge in $E$ that crosses this interval. In particular, it is an elementary consequence of \cite[Lemma 10.1]{CS2} that if $C_x$ is the closest cut-point to $x$ that lies on the left-hand side of $x$, then, $\mathbf{P}$-a.s., $n^{-1}C_{xn}\rightarrow x$ uniformly on compacts. It follows that, $\mathbf{P}$-a.s., uniformly over compact regions of $0\leq x\leq y$,
\begin{eqnarray*}
\lefteqn{\left|n^{-1}R(xn,yn)-R_\infty|x-y|\right|}\\
&\leq& \left|n^{-1}R(C_{xn},C_{yn})-R_\infty|x-y|\right|+2\max_{z\in\{x,y\}}n^{-1}\left|zn-C_{zn}\right|\\
&= &\left|n^{-1}R(C_{0},C_{yn})-n^{-1}R(C_{0},C_{xn})-R_\infty|x-y|\right|+2\max_{z\in\{x,y\}}n^{-1}\left|zn-C_{zn}\right|\\
&\leq&\left|n^{-1}R(0,yn)-n^{-1}R(0,xn)-R_\infty|x-y|\right|+6\max_{z\in\{0,x,y\}}n^{-1}\left|zn-C_{zn}\right|\\
&\rightarrow& 0,
\end{eqnarray*}
where to deduce the equality, we apply the series law for resistors, which clearly holds at cut-points. To complete the proof of \eqref{claims2}, it remains to check that $R_\infty>0$. Since $R(0,n)$ is bounded below by the number of cut-points between 0 and $n$, this can be deduced by another application of \cite[Lemma 10.1]{CS2}.

Moreover, we have that, $\mathbf{P}$-a.s.,
\begin{equation}\label{claims3}
\lim_{r\rightarrow\infty}\liminf_{n\rightarrow\infty}n^{-1}R\left(0,\{-rn,\dots,rn\}^c\right)=\infty.
\end{equation}
Indeed, if we define $C_0$ and $C_{rn}$ as above, and let $\tilde{C}_{-rn}$ be the closest cut-point to $-rn$ that lies on the right-hand side of $-rn$, then
\[R\left(0,\{-rn,\dots,rn\}^c\right)\geq R\left(C_0,\{C_{rn},\tilde{C}_{-rn}\}\right)-R(0,C_0)\geq \frac{1}{2}\min_{z\in\{C_{rn},\tilde{C}_{-rn}\}}R(0,z)-R(0,C_0),\]
where we have applied the parallel law to deduce the second inequality. Hence applying the conclusion of the previous paragraph for point-wise resistances yields that
\[\liminf_{n\rightarrow\infty}n^{-1}R\left(0,\{-rn,\dots,rn\}^c\right)\geq \frac{R_\infty r}{2},\]
which clearly implies \eqref{claims3}.

Putting \eqref{claims1} and \eqref{claims2} together similarly to the argument of \cite[Theorem A.1]{CFJ}, we obtain that
\[\left(\mathbb{Z},n^{-1}R,n^{-1}\tilde{\pi},0,n^{-1}I_{\mathbb{Z}}\right),\]
where $I_{\mathbb{Z}}$ is the identity map on $\mathbb{Z}$,  $\mathbf{P}$-a.s.\ converges in the spatial Gromov-Hausdorff-vague topology (see \cite[Section 7]{Croy} and \cite[Section 2.2]{CHK} for details) to
\[\left(\mathbb{R},R_\infty d_E,C\mathcal{L},0,I_{\mathbb{R}}\right),\]
where $d_E$ is the Euclidean metric, $\mathcal{L}$ is the Lebesgue measure on $\mathbb{R}$, and $I_{\mathbb{R}}$ is the identity map on $\mathbb{R}$. Together with the resistance divergence of \eqref{claims3}, this enables us to apply \cite[Theorem 7.1]{Croy} to deduce a Brownian motion scaling limit for the continuous-time version of $X$, with mean one exponential holding times. The result for the discrete-time process readily follows.
\end{proof}

\subsection{Long-range percolation beyond the nearest-neighbour case} \label{ss:ll}

As we noted in Remark \ref{rem:lq}, to go beyond the nearest-neighbour case in establishing heat kernel lower bounds via the approach of this article, it will suffice to check the conditions (V) and (BS). In this section, we describe some progress in this direction, which allows us to consider the non-nearest-neighbour model of Section \ref{sec3} for $d\geq 1$ and $s\in (d,2d)$.

For (V), the essential work was completed by Biskup in \cite{B}, where estimates on the size of a largest percolation cluster in a box were given. In the following lemma, we transfer the desired estimate to the vertex set $V_n$. (We recall that $G$ is the infinite cluster of the long-range percolation model, and $G_n=(V_n,E_n)$ is the largest connected component of $G \cap [-n,n]^d$.)

\begin{lem}\label{vlem} If $d\geq 1$ and $s\in (d,2d)$, then LRP($d$,$s$) satisfies (V).
\end{lem}
\begin{proof} Let $\mathcal{C}_1(n)$ be the largest connected component of  LRP($d$,$s$) inside $[-n,n]^d$. It is proved as \cite[Theorem 3.2]{B} that, for any $s'\in (s,2d)$, there exists a constant $\varepsilon>0$ such that
\begin{equation}\label{biskvol}
\mathbf{P}\left(\left|\mathcal{C}_1(n)\right|< \varepsilon n^d\right)\leq e^{-\varepsilon n^{2d-s'}}
\end{equation}
for all large $n$. We seek to replace $\mathcal{C}_1(n)$ by $V_n$ in this estimate, which we will do by showing that $\mathcal{C}_1(n)$ is a part of the infinite component with suitably high probability (cf.\ the proof of \cite[Corollary 3.3]{B}). To this end, for a given $n$, define $r_n^k:=2^kn$, and let $x_n^k$ be a sequence of points on the first coordinate axis such that $x_n^0=0$ and $|x_n^{k+1}-x_n^k|=3r_n^k$. In particular, the $\ell_\infty$-balls $B_\infty(x_n^k,r_n^k)$, $k\geq1$ are disjoint, but consecutive elements of the sequence touch. Write $\mathcal{C}_n^k$ for the largest connected component of  LRP($d$,$s$) inside $B_\infty(x_n^k,r_n^k)$, and define $\mathcal{C}_n^k\not\leftrightarrow\mathcal{C}_n^{k+1}$ to be the event that the two components in question are not connected by a direct edge. Applying \eqref{biskvol}, we then have that, for large $n$,
\begin{eqnarray*}
\mathbf{P}\left(\mathcal{C}_n^k\not\leftrightarrow\mathcal{C}_n^{k+1}\right)&\leq& 2e^{-\varepsilon (r_n^k)^{2d-s'}}+\mathbf{P}\left(\mathcal{C}_n^k\not\leftrightarrow\mathcal{C}_n^{k+1},\:|\mathcal{C}_n^k|\geq \varepsilon(r_n^k)^d,\:|\mathcal{C}_n^{k+1}|\geq \varepsilon(r_n^{k+1})^d\right)\\
&\leq &2e^{-\varepsilon (r_n^k)^{2d-s'}}+e^{- \varepsilon^2(r_n^k)^d(r_n^{k+1})^d/(2r_n^k+2r_n^{k+1})^s}\\
&\leq &Ce^{-2c (n2^k)^{2d-s'}}\\
&\leq &Ce^{-c (n^{2d-s'}+  2^{k(2d-s')})},
\end{eqnarray*}
where for the second inequality we have used the fact that the maximal distance between points in $\mathcal{C}_n^k$ and $\mathcal{C}_n^{k+1}$ is $2r_n^k+2r_n^{k+1}$, and for the last, we use that $a+b\leq 2ab$ for $a,b\geq 1$. Hence, writing $\mathcal{C}_n^k\leftrightarrow\mathcal{C}_n^{k+1}$ for the event that the two components in question are connected by a direct edge, we find that
\begin{eqnarray}
\mathbf{P}\left(V_n=\mathcal{C}_1(n)\right)&=&\mathbf{P}\left(\mathcal{C}_1(n)\subseteq G\right)\nonumber\\
&\geq &\mathbf{P}\left(\mathcal{C}_n^k\leftrightarrow\mathcal{C}_n^{k+1}\mbox{ for all }k\geq 0\right)\nonumber\\
&\geq &1-\sum_{k=0}^\infty Ce^{-c (n^{2d-s'}+  2^{k(2d-s')})}\nonumber\\
&=&1-Ce^{-c n^{2d-s'}}.\label{sameclust}
\end{eqnarray}
Putting this bound together with \eqref{biskvol}, we readily obtain (V).
\end{proof}

We next present sufficient conditions for the LRP($d$,$s$) model of Section \ref{sec3} to satisfy the Benjamini-Schramm convergence condition (BS). Roughly speaking, the first of the conditions we introduce means that there is only a small probability that a long path avoids the largest connected component in a box, and the second one implies a weak law of large numbers for the size of the component. (Clearly, both conditions are trivial in the nearest-neighbour case.)

\begin{lem}\label{bslem}
Suppose LRP($d$,$s$) (as defined in Section \ref{sec3}) satisfies the volume condition (V). Let $(a_n)_{n\geq 1}$ be a divergent sequence of positive integers such that $a_n=o(n)$, and suppose that
\begin{itemize}
\item[(a)] $\sup_{x \in W_n^*} \pp(x \in G \setminus W_n) =o(1)$,
\item[(b)] $\Var(|W_n|) =o(n^{2d})$,
\end{itemize}
where $W_n^*=[-n+a_n,n-a_n]^d\cap \Z^d$ and $W_n=W_n^*\cap V_n$. Then the random rooted graphs $(G_n,\rho_n)_{n\geq 1}$, with $\rho_n$ uniformly chosen in $V_n$, Benjamini-Schramm converge to $(G,\rho)$, conditioned that $\rho \in G$.
\end{lem}
\begin{proof} Assume that (V), (a) and (b) hold. We need to show that for any finite graph $H$ and $r \in \mathbb{N}$,
\begin{equation} \label{ll}
    \lim_{n\rightarrow \infty} \pp(B_{G_n}(\rho_n,r)=H) = \pp(B_{G}(\rho,r)=H\:|\:\rho\in G).
\end{equation}
First, we have
\begin{eqnarray*}
\pp(\rho_n \not \in W_n) = \E(|V_n\setminus W_n|/|V_n|) &\preceq&n^{-d} \E(|V_n\setminus W_n|) + \pp(|V_n| \leq cn^d)\\
&\preceq& n^{-d} n^{d-1} a_n +  \exp(-c(\log n)^2) =o(1).
\end{eqnarray*}
Therefore,
\begin{eqnarray*}
\pp[B_{G_n}(\rho_n,r)=H] &=& \pp[B_{G_n}(\rho_n,r)=H \mid \rho_n \in W_n]+o(1)\\
&=& \E \left(\frac{1}{|W_n|}\sum_{x \in W_n}\I(B_{G_n}(x,r)=H)\right)+o(1)\\
&=&\E \left(\frac{1}{|W_n|}\sum_{x \in W^*_n}\I(B_{G_n}(x,r)=H; x \in W_n)\right)+o(1).
\end{eqnarray*}
Moreover, by translation invariance,
\begin{eqnarray*}
\pp [B_{G}(\rho,r)=H\mid  \rho  \in G] &=& \frac{\pp \left(B_{G}(\rho,r)=H; \, \rho \in G \right)}{ \pp(\rho \in G)} \notag \\
&=& \frac{1}{|W_n^*|\pp(\rho \in G) } \E \left( \sum_{x \in W_n^*} \I(B_{G}(x,r)=H; \, x \in G) \right).
\end{eqnarray*}
It follows from the last two equations that
\begin{equation} \label{boa4}
    \left|\pp[B_{G_n}(\rho_n,r)=H] -\pp [B_{G}(\rho,r)=H\mid  \rho  \in G]\right| \leq   \left|\E\left(  \frac{A_n}{B_n} - \frac{A_n'}{B_n'}  \right) \right|+ o(1),
\end{equation}
where
\begin{eqnarray*}
&A_n  = \sum_{x \in W_n^*} \I(B_G(x,r)=H; \, x \in G),\qquad &B_n = |W_n^*|\pp(\rho \in G), \\
&A'_n= \sum_{x \in W_n^*} \I(B_{G_n}(x,r)=H; \, x \in W_n), \qquad &B'_n = |W_n|.
\end{eqnarray*}
Observe that
\begin{eqnarray}\label{ddd1}
\left| \E \left( \frac{A_n}{B_n} - \frac{A'_n}{B'_n} \right) \right|  \leq  \frac{1}{B_n} \E[|A_n-A_n'|] + \E\left( \frac{A_n'|B_n-B_n'|}{B_n B_n'} \right).
\end{eqnarray}
To bound the first term, note that
\begin{eqnarray*}
   \E |A_n-A_n'|&\leq& \E\sum_{x \in W_n^*}\left( \I(x \in G \setminus W_n)+\I(x \in G,\:B_G(x,r)\not\subseteq [-n,n]^d)\right)\\
   &\preceq & n^d \max_{x \in W_n^*} \left( \pp(x \in G \setminus W_n)+\pp(x \in G,\:B_G(x,r)\not\subseteq [x-a_n,x+a_n]^d)\right)\\
   &=&o(n^d)+n^d\pp(\rho \in G,\:B_G(\rho,r)\not\subseteq [-a_n,a_n]^d)\\
   &=&o(n^d),
\end{eqnarray*}
where we have used that $|W_n^*|\asymp n^d$ for the second inequality, (a) and translation invariance for the first equality, and the almost-sure finiteness of $B_G(\rho,r)$ (and the divergence of $(a_n)_{n\geq 1}$) for the second equality. Since $B_n\asymp n^d$, it follows that
\begin{equation}\label{ddd2}
\frac{1}{B_n} \E[|A_n-A_n'|]=o(1).
\end{equation}
Using that $A_n' \leq |W_n^*| \asymp B_n$, we further obtain that
\begin{eqnarray}
\E\left( \frac{A_n'|B_n-B_n'|}{B_n B_n'} \right) &\preceq&  \E\left( \frac{|B_n-B_n'|}{ B_n'} \right) =   \E\left( \frac{||W_n|-|W_n^*|\pp(\rho \in G)|}{ |W_n|} \right) \notag \\
&\preceq & n^{-d} \E[||W_n|-|W_n^*|\pp(\rho \in G)|] +  n^d\pp[|W_n| \leq cn^d] \notag \\
&\preceq& n^{-d} \left(\E \left(\left( |W_n|-|W_n^*|\pp(\rho \in G)\right)^2 \right) \right)^{1/2} + \exp(-c(\log n)^2),\label{ddd3}
\end{eqnarray}
by using the Cauchy-Schwarz inequality and (V) (and the fact that $|V_n|\preceq |W_n|+n^{d-1}a_n= |W_n|+o(n^{d})$). Additionally,
\begin{eqnarray} \label{covw}
\lefteqn{\E \left(\left( |W_n|-|W_n^*|\pp(\rho \in G)\right)^2 \right) = \E \left(\left( \sum_{x \in W_n^*} (\I(x \in W_n) - \pp(x \in G))\right)^2 \right) }\notag \\
&=& \E \left(\left( \sum_{x \in W_n^*} (\I(x \in W_n) - \pp(x \in W_n) - \pp(x \in G \setminus W_n))\right)^2 \right) \notag \\
&\leq& 2\E \left(\left( \sum_{x \in W_n^*} \I(x \in W_n) - \pp(x \in W_n)  \right)^2 \right) +2 |W_n^*|^2 \max_{x\in W_n^*}\pp[x \in G \setminus W_n] \notag \\
&=& 2 \Var[|W_n|]+ 2 |W_n^*|^2 \max_{x\in W_n^*}\pp[x \in G \setminus W_n] \notag \\
&=& o(n^{2d}),
\end{eqnarray}
by using (a) and (b). Combining the estimates \eqref{boa4}--\eqref{covw}, and using (a) again, we obtain \eqref{ll}.
\end{proof}

In the subsequent lemma, we apply Lemma \ref{bslem} for $d\geq 1$ and $s\in (d,2d)$. We highlight that the proof depends on the estimates \eqref{biskvol} and \eqref{sameclust} from the proof of Lemma \ref{vlem}, and also two further statements from \cite[Theorem 2]{CS}. Inspecting the latter reference, one would find that \cite[Theorem 2]{CS} is stated for the smaller range $s\in (d,\min\{d+2,2d\})$. However, as is commented below \cite[Theorem 2]{CS} (and a careful checking of the argument establishes), this restriction, which is principally due to the main focus of the paper being on the stable regime, is not essential, and the percolation estimates of \cite[Theorem 2]{CS} extend to the range $s\in (d,2d)$.

\begin{lem}\label{bschecklem} If $d\geq1$ and $s\in(d,2d)$, then LRP($d$,$s$) satisfies (BS).
\end{lem}
\begin{proof} We will check the two conditions of Lemma \ref{bslem}.

Towards verifying (a), let $x\in W_n^*$, and $\mathcal{C}_1(x,a_n)$ be the largest connected component of LRP($d$,$s$) in the $\ell_\infty$-ball $B_\infty(x,a_n)$. We first claim that if $a_n=(\log n)^\Delta$ for suitably large $\Delta$, then
\begin{equation}\label{firstb}
\max_{x\in W_n^*}\mathbf{P}\left(\mathcal{C}_1(x,a_n)\not\subseteq V_n\right)=o(1).
\end{equation}
Letting $\mathcal{C}_1(n)$ and $\mathcal{C}_2(n)$ be the first and second largest components of LRP($d$,$s$) inside $[-n,n]^d$, respectively, we have that, for any $\varepsilon>0$,
\begin{eqnarray*}
\mathbf{P}\left(\mathcal{C}_1(x,a_n)\not\subseteq V_n\right)&\leq&
\mathbf{P}\left(|\mathcal{C}_1(x,a_n)|<\varepsilon a_n^d\right)+\mathbf{P}\left(\mathcal{C}_1(x,a_n)\not\subseteq V_n,\:|\mathcal{C}_1(x,a_n)|\geq \varepsilon a_n^d\right)\\
&\leq &\mathbf{P}\left(|\mathcal{C}_1(a_n)|<\varepsilon a_n^d\right)+\mathbf{P}\left(|\mathcal{C}_2(n)|\geq \varepsilon a_n^d\right)+\mathbf{P}\left(V_n\neq \mathcal{C}_1(n)\right),
\end{eqnarray*}
where to obtain the second inequality we use that on the event $\{V_n=\mathcal{C}_1(n),\:\mathcal{C}_1(x,a_n)\not\subseteq V_n\}$, it holds that $|\mathcal{C}_2(n)|\geq |\mathcal{C}_1(x,a_n)|$. Moreover, we highlight that we have used the translation invariance of the model to derive a bound that does not depend on the choice of $x\in W_n^*$. By \eqref{biskvol} and \eqref{sameclust}, the first and third probabilities above are $o(1)$ as $n\rightarrow\infty$. As for the second term, this is shown to be $o(1)$ in \cite[Theorem 2]{CS} for $\Delta$ chosen large enough. Hence we have established \eqref{firstb}.

Applying the estimate of previous paragraph, we find that
\begin{eqnarray*}
\max_{x\in W_n^*}\mathbf{P}\left(x\in G\backslash W_n\right)&\leq &o(1)+\max_{x\in W_n^*}\mathbf{P}\left(x\in G\backslash \mathcal{C}_1(x,a_n)\right)\\
&\leq &o(1)+\max_{x\in W_n^*}\mathbf{P}\left(x\leftrightsquigarrow B_\infty(x,a_n)^c,\: x\not\in \mathcal{C}_1(x,a_n)\right)\\
&=&o(1)+\mathbf{P}\left(0\leftrightsquigarrow B_\infty(0,a_n)^c,\: 0\not\in \mathcal{C}_1(a_n)\right)
\end{eqnarray*}
where $x\leftrightsquigarrow B_\infty(x,a_n)^c$ means that $x$ is connected to the complement of $B_\infty(x,a_n)$, but not necessarily by a single edge. The probability in the final line above is shown to be $o(1)$ in \cite[Theorem 2]{CS}. This confirms (a).

For proving condition (b) of Lemma \ref{bslem}, we first observe
\[\Var(|W_n|)\preceq (na_n)^d +\sum_{\substack{x,y\in W_n^*:\\|x-y|>2a_n}}\mathrm{Cov}(\mathbb{I}(x\in W_n),\mathbb{I}(y\in W_n)).\]
Now, some elementary manipulation of probabilities allows it to be checked that, for $x,y\in W_n^*$,
\begin{eqnarray*}
\mathrm{Cov}(\mathbb{I}(x\in W_n),\mathbb{I}(y\in W_n))&\leq& \mathrm{Cov}(\mathbb{I}(x\in \mathcal{C}_1(x,a_n)),\mathbb{I}(y\in  \mathcal{C}_1(y,a_n)))\\
&&+2\max_{z\in W_n^*}\mathbf{P}\left(z\in \mathcal{C}_1(z,a_n)\backslash W_n \right)+2\max_{z\in W_n^*}\mathbf{P}\left(z\in W_n\backslash \mathcal{C}_1(z,a_n)\right).
\end{eqnarray*}
Moreover, if $|x-y|> 2a_n$, then the first term in the upper bound is zero. To bound the first of the probabilities, note that if the event $z\in  \mathcal{C}_1(z,a_n)\backslash W_n$ occurs, then it must be the case that $\mathcal{C}_1(z,a_n)\not\subseteq V_n$. Thus, by \eqref{firstb},
\[\max_{z\in W_n^*}\mathbf{P}\left(z\in \mathcal{C}_1(z,a_n)\backslash W_n \right)\leq\max_{z\in W_n^*}\mathbf{P}\left(z\in \mathcal{C}_1(z,a_n)\not\subseteq V_n \right)=o(1).\]
Additionally, since $W_n\subseteq V_n\subseteq G$, we have that
\[\max_{z\in W_n^*}\mathbf{P}\left(z\in W_n\backslash \mathcal{C}_1(z,a_n)\right)\leq \max_{z\in W_n^*}\mathbf{P}\left(z\in G\backslash \mathcal{C}_1(z,a_n)\right).\]
That the upper bound is $o(1)$ was established earlier in the proof. Combining the previous estimates yields (b), as desired.
\end{proof}

From Remark \ref{rem:lq} and Lemmas \ref{vlem} and \ref{bschecklem} (and the heat kernel upper bound of \cite{CS}), we obtain the following result for the long-range percolation model of Section \ref{sec3}, which we underline does not require nearest-neighbour connections. In particular, this result verifies that the on-diagonal heat kernel estimates of \cite{CS} are sharp (up to logarithmic factors) throughout the stable regime $s\in(d, \min\{d+2,2d\})$ (and not just when $s\in (d,d+1)$, which is what is implied by the invariance principle of \cite{CS2}).

\begin{cor}\label{corcorcor} If $d\geq1$ and $s\in(d, 2d)$, then LRP($d$,$s$) satisfies the relevant lower heat kernel bounds of Theorems \ref{lrpq} and \ref{lrpa}. If $d\geq1$ and $s\in(d, \min\{d+2,2d\})$, then it also satisfies the spectral dimension result of Corollary \ref{corlrp}.
\end{cor}

\section*{Acknowledgments}

This research was partially supported by JSPS KAKENHI, grant numbers 17F17319, 17H01093 and 19K03540, by the Singapore Ministry of Education Academic Research Fund Tier 2 grant number MOE2018-T2-2-076, by the Vietnam Academy of Science and Technology grant number CTTH00.02/22-23, and by the Research Institute for Mathematical Sciences, an International Joint
Usage/Research Center located in Kyoto University.

\bibliographystyle{amsplain}
\bibliography{SRWoLRP}

\end{document}